\documentclass[pdflatex]{sn-jnl}
\usepackage{multirow}%
\usepackage{amsmath,amssymb,amsfonts}%
\usepackage{amsthm}%
\usepackage{mathrsfs}%
\usepackage[title]{appendix}%
\usepackage{xcolor}%
\usepackage{textcomp}%
\usepackage{manyfoot}%
\usepackage{booktabs}%
\usepackage{algorithm}%
\usepackage{algorithmicx}%
\usepackage{algpseudocode}%
\usepackage{listings}%

\usepackage{bm, mathdots, url}

\DeclareMathOperator{\app}{app}

\DeclareMathOperator{\intt}{int}
\DeclareMathOperator{\pos}{pos}
\DeclareMathOperator{\parr}{par}
\DeclareMathOperator{\supp}{supp}

\newcommand{\R}{\mathbb{R}}
\newcommand{\Q}{\mathbb{Q}}
\newcommand{\Z}{\mathbb{Z}}
\newcommand{\N}{\mathbb{N}}
\newcommand{\bA}{\bm{A}}
\newcommand{\bx}{\bm{x}}
\newcommand{\by}{\bm{y}}
\newcommand{\bb}{\bm{b}}
\newcommand{\bv}{\bm{v}}
\newcommand{\bd}{\bm{d}}
\newcommand{\ba}{\bm{a}}
\newcommand{\bz}{\bm{z}}
\newcommand{\br}{\bm{r}}
\newcommand{\be}{\bm{e}}
\newcommand{\bB}{\bm{B}}
\newcommand{\bD}{\bm{D}}
\newcommand{\bU}{\bm{U}}

\theoremstyle{thmstyleone}%
\newtheorem{theorem}{Theorem}%
\newtheorem{proposition}{Proposition}
\newtheorem{lemma}{Lemma}
\newtheorem{corollary}{Corollary}

\theoremstyle{thmstyletwo}%
\newtheorem{example}{Example}%

\theoremstyle{thmstylethree}%

\begin{document}
	\title[Sparse Approximation in Lattices and Semigroups]{Sparse Approximation in Lattices and Semigroups}
	
	\author[1]{\fnm{Stefan} \sur{Kuhlmann}}\email{stefan.kuhlmann@math.ethz.ch}
	
	\author[2]{\fnm{Timm} \sur{Oertel}}\email{timm.oertel@fau.de}
	
	\author[1]{\fnm{Robert} \sur{Weismantel}}\email{robert.weismantel@ifor.math.ethz.ch}
	
	\affil[1]{\orgdiv{Department of Mathematics}, \orgname{ETH Zürich}, \city{Zürich}, \country{Switzerland}}
	
	\affil[2]{\orgdiv{Department of Mathematics}, \orgname{FAU Erlangen-Nürnberg}, \city{Erlangen}, \country{Germany}}
	
	\abstract{This paper deals with the following question: Suppose that there exist an integer or a non-negative integer solution $\bx$ to a system $\bA\bx = \bb$, where the number of non-zero components of $\bx$ is $n$. The target is, for a given natural number $k < n$, to approximate $\bb$ with $\bA\by$ where $\by$ is an integer or non-negative integer solution with at most $k$ non-zero components. 
	We establish upper bounds for this question in general. In specific cases, these bounds are tight. If we view the approximation quality as a  function of the  parameter $k$, then  the paper explains why the quality of the approximation increases exponentially as $k$ goes to $n$. This paper is a complete version of an extended abstract that appeared at the 26th International Conference on Integer Programming and Combinatorial Optimization (IPCO) \cite{ipco_version}.}
	
	\keywords{Integer Programming, Sparse Integral Solutions, Semigroups, Approximate Carath\'eodory}
	
	\maketitle
	
	\section{Introduction}
	The central object in this paper are sparse integer-valued vectors. A vector is called sparse if only few of its entries are non-zero. In the context of an integer linear program with constraints $\bA\bx=\bb, \bx\in\Z^n_{\geq 0}$, sparse optimal vectors are often preferred over general solutions, as each non-zero entry typically corresponds to a change in the underlying system modeled by the integer linear program. Having only a few of those changes can reduce costs, ensure robustness, and make it easier to interpret the optimal solution. In recent years, several results have been proven that guarantee sparse integer solutions to $\bA\bx=\bb, \bx\in\Z^n_{\geq 0}$ for integer-valued $\bA$ and $\bb$. These bounds only depend on the number of equations and the size of the matrix $\bA$; see Subsection~\ref{ss_related_literature} for the explicit bounds and references. While these bounds already imply that the number of non-zero entries of an optimal solution cannot be too large, the number of non-zeros may still be impractical. This might happen if, for instance, the entries in $\bA$ are large. Motivated by this, we study the following trade-off between sparsity and feasibility: given a sparsity bound $k$, can we find $\bx\in\Z^n_{\geq 0}$ such that $\bx$ has at most $k$ non-zero entries and $\bA\bx \approx \bb$?

	Let us make this mathematically clean. 
	Fix a set $X\subseteq \R^n$. We will consider the cases where $X = \Z^n$ is the underlying lattice or $X=\Z_{\geq 0}^n $ is the underlying semigroup. 
	The second input to the problem is a matrix
	$\bA = (\ba_1,\ldots,\ba_n)\in\Z^{m\times n}$ of full row rank. We are interested in approximations of elements in the set $\bA\cdot X := \lbrace \bb \in \R^m : \bb = \bA\bx \text{ for }\bx\in X\rbrace$ by elements of $X$ whose support is small. In order to turn this into a formula, let $\Vert\cdot\Vert$ be a norm on $\R^m$ and let $k\in\lbrack n\rbrack$ be a given sparsity parameter, where $\lbrack n\rbrack :=\lbrace 1,\ldots,n\rbrace$ for $n\geq 1$. Our target is to determine the best possible bound in terms of $n,m,k$, and eventually some parameter associated with $\bA$ for the problem 
	\begin{align*}
		\app_{X,k,\Vert\cdot\Vert}(\bA) := \max_{\bb\in \bA\cdot X}\min \left\lbrace \Vert \bA\bx - \bb\Vert : \bx\in X, \ \left|\supp(\bx)\right|\leq k\right\rbrace,
	\end{align*}
	where $\supp(\bx) :=\lbrace i\in\lbrack n\rbrack : \bx_i\neq 0\rbrace$ for $\bx\in\R^n$. 
	This is a very broad definition and it comes as no surprise that bounds for this problem vary significantly with different input parameters $X$ and $\Vert\cdot\Vert$. 
	
	We begin by discussing the parameter $\app_{X,k,\Vert\cdot\Vert}(\bA)$ in the continuous setting when $X = \R^n$ and $X = \R^n_{\geq 0}$. If $X = \R^n$ and $k\geq m$, then we have
	$\app_{X,k,\Vert\cdot\Vert}(\bA) = 0$, since every $\bb\in \bA\cdot X = \R^m$ can be expressed by an invertible $m\times m$ submatrix of $\bA$. 
	Similarly, when $X = \R^n_{\geq 0}$, a result by Carath\'eodory implies that $\app_{X,k,\Vert\cdot\Vert}(\bA) = 0$ if $k\geq m$. For $k < m$, in both cases, $X = \R^n$ and $X = \R^n_{\geq 0}$, finding an optimal sparse approximation is essentially equivalent to identifying a $k$-dimensional subspace spanned by a subset of the set of columns of $\bA$ that is close to a target vector $\bb$. Of course, if we maximize over all
	feasible vectors in $\bA\cdot X$, the approximation error
	$\app_{X,k,\Vert\cdot\Vert}(\bA)$ then diverges to infinity as $\bb$ may lie
	arbitrarily far away from any of these $k$-dimensional subspaces. 
	This argument remains true when $X = \Z^n$ or $X = \Z^n_{\geq 0}$.
	To avoid this situation, we assume throughout that $k \geq m$.
	
	\subsection{Related literature}
	\label{ss_related_literature}
	To the best of our knowledge, there are no studies of the parameter $\app_{X,k,\Vert\cdot\Vert}(\bA)$ as a function in $k$ available when $X = \Z^n$ or $X = \Z^n_{\geq 0}$. 
	In contrast, the task of finding an upper bound on $k$ that guarantees $\app_{X,k,\Vert\cdot\Vert}(\bA) = 0$ has received considerable attention. 
	This problem is equivalent to finding a sparse solution to the equality system $\bA\bx =\bb$ with $\bx\in\Z^n$ or $\bx\in\Z^n_{\geq 0}$. When $X=\Z^n$, such an upper bound on $k$ has been established in \cite[Theorem 1]{alievAverkovLoeraOertel21}: Let
	\begin{align}
		\label{def_small_delta}
		\delta(\bA) := \min\lbrace \left|\det\bB\right| : \bB \text{ is an invertible }m\times m\text{ submatrix of }\bA\rbrace.
	\end{align}
	Then the bound in \cite[Theorem 1]{alievAverkovLoeraOertel21} gives essentially
	$\app_{\Z^n,k,\Vert\cdot\Vert_\infty}(\bA) = 0$ whenever $k \geq m + \log_2(\delta(\bA))$. More refined results, which use among others the prime factorization of $\delta(\bA)$, are given in \cite{alievAverkovLoeraOertel21,DubeyLiu2025}. These bounds are asymptotically tight. 
	When $X =\Z^n_{\geq 0}$, it has been shown in \cite{eisenbrandshmonincaratheodorybounds06} that $\app_{X,k,\Vert\cdot\Vert}(\bA) = 0$ whenever 
	$k \geq 2m\log_2(4m\Vert\bA\Vert_\infty)$, where $\Vert\bA\Vert_\infty$ denotes the largest entry of $\bA$ in absolute value. This bound has been improved to $2m\log_2(2\sqrt{m}\Vert\bA\Vert_\infty)$ in \cite{alievdeloesparelindio2017}, which is tight up to the constants; see \cite{berndt2021new,berndt2024new}.
	Upper bounds on the sparsity are also available when the parameter $\Vert\bA\Vert_\infty$ is replaced by $\sqrt{\det\bA\bA^\top}$ or the largest subdeterminant of $\bA$ \cite{alievAverkovLoeraOertel21,alievdeloesparelindio2017,gribanov2024delta,lee2020improving}. 
	All these bounds apply also to optimal integral solutions of integer linear programs as it has been shown in \cite{alideloeisoerweissupportint2018}.
	
	The task of determining $\app_{X,k,\Vert\cdot\Vert}(\bA)$ has the flavor of a sparse recovery problem in the context of signal processing. One typical high-level question in this area is to find conditions on $\bA$ such that one can replace the constraint $\left|\supp(\bx)\right|\leq k$ by $\|\bx\|_1 \leq k$ when $k$ is small, i.e., $k \in\mathcal{O}(\log_2(m))$. A fundamental assumption for this to work  is to know that a sparse solution exists, or even that it is unique. Then one can try to recover it by applying a tractable algorithm. This line of research is very active, with extensive literature addressing the cases when $X$ contains real-valued or complex-valued vectors; see for instance \cite{Candes2006Robust,Candes2006Stable,Candes2005Decoding,Candes2007Dantzig,Donoho2006Compressed} or \cite{foucart_rauhut_2013} for a comprehensive overview. In our work, we are concerned with the case when $X$ contains integer-valued vectors; see \cite{alasmari2023unique,flinth2018promp,fukshansky2019algebraic,konyagin2018recovery} for some results on an integer version of sparse recovery. 
	The direction that we take in this paper deviates from the classical research questions addressed in the field of sparse recovery in signal processing. In particular, we consider a fixed matrix $\bA$ and have no a-priori knowledge about the existence of a sparse solution.
	
	Let us remark that investigating $\app_{X,k,\Vert\cdot\Vert}(\bA)$ when $k<m$ and $X = \R^n_{\geq 0}$ can be seen as an approximate version of Carath\'eodory's theorem. Variants of this problem have been studied in different contexts such as faster algorithms for integer programming \cite{approx_to_exact}, approximating Nash equilibria \cite{nash_equilibria}, and deterministically factorizing sparse polynomials with bounded degrees \cite{deterministic_factorization}, to name a few; see \cite{pmlr-v70-mirrokni17a} for tight bounds. Our results can be viewed as an extension of this line of work to the integer setting $\Z^n_{\geq 0}$.
	
	\subsection{Contributions}
	We begin with $X=\Z^n$. Recall from above that the bound in \cite[Theorem 1]{alievAverkovLoeraOertel21} already gives us $\app_{\Z^n,k,\Vert\cdot\Vert_\infty}(\bA) = 0$ whenever $k \geq m + \log_2(\delta(\bA))$. 
	In this paper, we show a more general result in Theorem~\ref{thm_lattice_main} as it states that 
	\begin{align*}
		\app_{\Z^n,k,\Vert\cdot\Vert_\infty}(\bA) \leq \delta(\bA)/2^{k-m+1},
	\end{align*}
	i.e., the approximation error decreases exponentially as $k$ goes to $n$. 
	This bound implies a sparsity bound on exact solutions: For $k \geq m + \log_2(\delta(\bA))$, we obtain $\app_{\Z^n,k,\Vert\cdot\Vert_\infty}(\bA)  \leq  1/2$. Since $\app_{\Z^n,k,\Vert\cdot\Vert_\infty}(\bA)$ is an integer, we conclude that $\app_{\Z^n,k,\Vert\cdot\Vert_\infty}(\bA)  = 0$. The bound $ m + \log_2(\delta(\bA))$ is of the same order of magnitude as the one presented in \cite[Theorem 1]{alievAverkovLoeraOertel21}. This is not a coincidence as our techniques extend the approach used there. Another new contribution is presented in Lemma~\ref{lemma_lattice_covering}. It provides an upper bound on the covering radius of the cube with respect to sublattices of $\Z^n$. This is a key step to derive Theorem~\ref{thm_lattice_main}.
	
	We turn our attention to the case $X = \Z^n_{\geq 0}$. Let $\pos Y$ denote the set of all non-negative combinations of elements in $Y\subseteq\R^m$. 
	In the special case when $\pos\lbrace \ba_1,\ldots,\ba_n\rbrace=\R^m$, the question of determining $\app_{\Z^n_{\geq 0},k,\Vert\cdot\Vert}(\bA)$ can be tackled similar to the lattice setting where $X =\Z^n$. This is made precise in Corollary \ref{cor_pos_spans_space}. 
	When $\pos\lbrace\ba_1,\ldots,\ba_n\rbrace$ is a pointed cone, the analysis of the quantity $\app_{\Z^n_{\geq 0},k,\Vert\cdot\Vert}(\bA)$ is significantly more involved. To make this problem more tractable, we focus on the case when $\pos\lbrace\ba_1,\ldots,\ba_n\rbrace$ is a pointed simplicial cone. 
	More precisely, we fix a basis $\bB$ consisting of linearly independent column vectors of $\bA$, say
	$\ba_1,\ldots,\ba_m$, and suppose that
	$\ba_i\in\pos\lbrace \ba_1,\ldots,\ba_m\rbrace$ for each $i\in\lbrack
	n\rbrack$. 
	Our analysis uses the norm $\Vert\cdot\Vert_{P(\bB)}$ induced by the symmetric
	parallelepiped $P(\bB) := \bB\cdot \lbrack -1,1 \rbrack^m$. 
	The results also depend on the parameter $\mu :=\max_{i\in\lbrack n\rbrack}\Vert\ba_i\Vert_{P(\bB)}$. Such a dependence is necessary to show that the approximation error decreases when $k$ increases; cf. Proposition~\ref{prop_semigroup_lower_bound_general}. 
	Then we show 
	\begin{align*}
		\app_{\Z^n_{\geq 0}, k, \Vert\cdot\Vert_{P(\bB)}}(\bA)\leq \frac{1}{2^{\frac{1}{m}} - 1}\cdot \left(\frac{1}{2^{\frac{k-m}{m}}} - \frac{1}{2^{\frac{n-m}{m}}}\right)\cdot \mu^{\frac{m-1}{m}}\cdot \left|\det\bB\right|^{\frac{m-1}{m}}
	\end{align*}
	in Theorem~\ref{thm_semigroup_general_bound}. When $m=1$, this bound becomes
	\begin{align*}
		\app_{\Z^n_{\geq 0}, k, |\cdot|}(\bA)\leq \left(\frac{1}{2^{k-1}} - \frac{1}{2^{n-1}}\right)\cdot a_1;
	\end{align*} 
	see Theorem~\ref{thm_semigroup_knapsack}. In addition to this, we present a tight upper bound when $m = 1$ and $k = 2$ in Theorem~\ref{thm_semigroup_approx_2}, which relies on Sylvester's sequence from number theory. We conclude with several instances with large approximation error in Section~\ref{s_semigroups_lower_bounds}.
	
	Some of the results in this paper appeared in the proceedings of the 26th Conference on Integer Programming and Combinatorial Optimization (IPCO) in an extended abstract \cite{ipco_version}. 
	This paper distinguishes itself from the earlier version. 
	Besides several additional explanations and proofs that were omitted, 
	two major changes are listed below:
	\begin{itemize}
		\item[(a)] If $\pos\lbrace \ba_1,\ldots,\ba_n\rbrace = \R^m$, then $\app_{\Z_{\geq 0}^n,k,\Vert\cdot\Vert }(\bA)$ can be determined with the tools developed for the analysis when $X=\Z^n$; see Corollary~\ref{cor_pos_spans_space}.
		\item[(b)] We provide instances for which the order of magnitude in $k$ matches our upper bounds; see Section~\ref{s_lattices}, Proposition~\ref{prop_semigroup_lower_bound_m_1_n-1}, and Proposition~\ref{prop_semigroup_lower_bound_m_1_n_2}. 
		This requires us to use results from elementary number theory. 
	\end{itemize}
	
	\section{Lattices and Cones that Span $\R^m$}
	\label{s_lattices}
	Let $X = \Z^n$ and $\bA\in\Z^{m \times n}$ have full row rank. The set $\bA\cdot X = \bA\Z^n$ is the lattice generated by the columns of $\bA$. We refer the reader to \cite{lekkerkerker2014geometry} for more about lattices, sublattices, and their determinants. Our aim is to find sparse approximations of lattice vectors in $\bA\Z^n$. We measure the approximation error with respect to the $\ell_\infty$-norm. Observe that if $k\geq m$, we can always express $\bb$ using a basis of the lattice $\bA\Z^n$, which consists of $m$ elements. However, it is not necessarily given that such a basis exists among the columns of $\bA$. Even stronger, it could be that no proper subset of the set of columns is a generating set for the lattice $\bA\Z^n$. In this situation, our approximation result below applies. Recall from (\ref{def_small_delta}) that $\delta(\bA)$ denotes the smallest subdeterminant in absolute value of an invertible $m\times m$ submatrix of $\bA$. 
	\begin{theorem}
		\label{thm_lattice_main}
		Let $\bA\in\Z^{m\times n}$ have full row rank and $m\leq k\leq n$. Then there exists a $m\times k$ submatrix $\bD$ of $\bA$ such that 
		\begin{align*}
			\app_{\Z^n,k,\Vert\cdot\Vert_\infty}(\bA) \leq \max_{\bb\in \bA\Z^n}\min_{\bx\in\Z^k}\Vert \bD\bx - \bb\Vert_\infty \leq \frac{1}{2^{k - m + 1}}\cdot \delta(\bA).
		\end{align*}
	\end{theorem}
	The first inequality holds by definition. 
	To show the second inequality, we rely on the fact that proper subsets of the set of columns of $\bA$ still generate some sublattice of $\bA\Z^n$. 
	Our point of departure is the following lemma that, in particular, holds for $m\times l$ submatrices of $\bA$ with full row rank. 
	\begin{lemma}
		\label{lemma_lattice_covering}
		Let $\bD\in\Z^{m\times l}$ have full row rank and $\Lambda =\bD \Z^l$. Then
		\[
		\min_{\bx\in \Z^l}\|\bD\bx-\bd\|_\infty  \le \frac{1}{2}\cdot\det\Lambda
		\]
		for all $\bd\in \R^m$.
	\end{lemma}
	\begin{proof}
		There exists a unimodular matrix $\bU\in\Z^{l\times l}$ such that $\bD\bU= [\bB, \bm{0}]$ is in Hermite normal form; cf.  \cite[Chapter 4]{schrijvertheorylinint86}.
		In particular, the matrix $\bB$ is a lower triangular matrix and $\det\Lambda=B_{1,1}\cdot\ldots\cdot B_{m,m}$, where $B_{i,i}$ refers to the $i$-th diagonal entry of $\bB$ for $i\in\lbrack m\rbrack$. 
		Define $\bx = (x_1,\ldots,x_l)^\top\in\Z^l$ and $\bz = (z_1,\ldots,z_m)^\top\in\Z^m$ recursively such that  
		\begin{align*}
			x_i = z_i := \left\lceil \frac{d_i-\sum_{j=1}^{i-1}B_{i,j}z_j}{B_{i,i}}\right\rfloor,
		\end{align*}
		for $i=1,\ldots,m$,
		where $\left\lceil y\right\rfloor$ denotes the integer closest to $y\in\R$, and $x_i=0$ for $i=m+1,\ldots,l$.
		Then we have
		\begin{align*}
			\left|(\bD\bU\bx - \bd)_i\right| = \left|(\bB\bz - \bd)_i\right| 
			= \left| B_{i,i}\left(\frac{ \sum_{j=1}^{i-1} B_{i,j}z_j -d_i}{B_{i,i}} + z_i\right)\right| \leq \tfrac{1}{2}B_{i,i}
		\end{align*}
		for all $i\in\lbrack m\rbrack$. The last term above is upper bounded by $1/2\cdot\det\Lambda$. So the integer vector $\bU\bx$ satisfies the claimed bound.
	\end{proof}
	Note, one may also interpret $\min_{\bx\in \Z^l}\|\bD\bx-\bd\|_\infty$ in Lemma~\ref{lemma_lattice_covering}  as the inhomogeneous minimum or the covering radius of the unit cube $[-1,1]^m$ with respect to the lattice $\Lambda$. Upper bounds on the inhomogeneous minima for general convex bodies typically rely on the so-called successive minima. In the special case of the unit cube, our bound surpasses these general results. Furthermore, the bound in Lemma~\ref{lemma_lattice_covering} is tight in general, which follows from the special case when $\Lambda = \Z^n$. 
	We refer the interested reader to \cite[Chapter 2, Section 13]{lekkerkerker2014geometry} for details concerning inhomogeneous minima, successive minima, and their relation.
	\begin{proof}[Proof of Theorem~\ref{thm_lattice_main}]
		We prove the second inequality by recursively constructing submatrices $\bD$ of $\bA$ to which Lemma \ref{lemma_lattice_covering} applies. Let $\bB$ be an $m\times m$ submatrix of $\bA$ such that $\left|\det\bB\right| = \delta(\bA)$. We set $\bD_0 := \bB$ and $\Lambda_0 := \bD_0\Z^m$. For $i\in \lbrack n - m\rbrack$, choose a column of $\bA$, say $\bv_i$, such that $\bv_i\notin \Lambda_{i-1}$. If there is no such column, we simply select an arbitrary column. Then we define $\bD_i := [\bD_{i-1},\bv_i]\in\Z^{m + i}$ and $\Lambda_i :=\bD_i\Z^{m + i}$. By Lemma~\ref{lemma_lattice_covering}, we have 
		\begin{align}
			\label{inequ_lattice_proof}
			\min_{\bx\in\Z^{m + i}}\Vert \bD_i\bx - \bb\Vert_\infty\leq \frac{1}{2}\cdot\det\Lambda_i	
		\end{align}
		for all $i\in\lbrace 0,\ldots,n - m\rbrace$ and $\bb \in \bA\Z^n$. To finish the proof, we analyze $\det\Lambda_i$ for $i\in\lbrace 0,\ldots,n-m\rbrace$. If $i=0$, we get $\det\Lambda_0 = \left|\det\bB\right| = \delta(\bA)$ and the claim follows. Fix some $i\in\lbrack n-m\rbrack$ and recall that $\bv_i$ is the column that we added in the $i$-th step. If $\bv_i\in \Lambda_{i-1}$, our construction implies that all columns of $\bA$ that are not columns of $\bD_{i-1}$ are contained in $\Lambda_{i-1}$. It follows that $\det\bA\Z^n = \det\Lambda_{i-1} = \det\Lambda_i$ and $\min_{\bx\in\Z^{m + i}}\Vert \bD_i\bx - \bb\Vert_\infty = 0$ for all $\bb\in\bA\Z^n$, which implies the claim. So suppose that $\bv_i\notin\Lambda_{i-1}$. This is equivalent to $\Lambda_{i-1}\subsetneq \Lambda_i$. Since $\Lambda_{i-1}$ is a proper sublattice of $\Lambda_i$, we obtain that $\det\Lambda_i \neq \det\Lambda_{i-1}$ and $\det\Lambda_i$ divides $\det\Lambda_{i-1}$. Thus, we have $\det\Lambda_i\leq 1/2\cdot\det\Lambda_{i-1}$. By recursively applying this inequality, we recover 
		\begin{align*}
			\det\Lambda_i \leq \frac{1}{2^i}\cdot\det\Lambda_0 = \frac{1}{2^i}\cdot\delta(\bA).
		\end{align*}
		Plugging this into (\ref{inequ_lattice_proof}) and choosing $i = k - m$ proves the claim. 
	\end{proof}
	
	Below is one example that shows that the bound presented in Theorem~\ref{thm_lattice_main} is tight when $k = n - 1$ and another example that illustrates the exponential decrease in $k$.
	
	\begin{example}
		Consider the matrix $\bA\in\Z^{m\times 2m}$ with columns  $2\be_i$ and $3\be_i$, where $\be_i$ denotes the $i$-th standard unit vector and $i=1,\ldots,m$.
		We have $\delta(\bA)=2^m$ and $\bA\Z^{2m}=\Z^m$.
		Noting that we need all columns to express the all-ones vector exactly we conclude that $\app_{\Z^{2m},2m-1,\Vert\cdot\Vert_\infty}(\bA)= 1 = 1/2^m\cdot\delta(\bA)$.
	\end{example}
	
	\begin{example}
		\label{ex_lattice_exponential_behavior}
		Let $p_1<\cdots<p_n$ be prime numbers and $q_i :=\prod_{j\neq i}p_j$. 
		Define the matrix 
		\begin{align*}
			\bA := (q_1 \be_1,q_2 \be_1,\ldots,q_n \be_1,\be_2,\be_3,\ldots,\be_m)\in\Z^{m\times (m+n-1)}.
		\end{align*}
		For a subset $I\subseteq[n]$, let $\bD_I$ denote the submatrix of $\bA$ consisting of the columns $q_i\be_1$ for $i\in I$ and $\be_2,\ldots,\be_m$. The construction of the $q_i$ ensures that, for any such $I\subseteq[n]$, $\gcd(\{q_i:i\in I\})=\prod_{i\in[n]\setminus I}p_i$.
		Using this property, one can show that $\det(\bD_I\Z^{|I|+m-1})=\prod_{i\in[n]\setminus I}p_i$ and
		a basis of $\bD_I\Z^{|I|+m-1}$ is given by $\lbrace(\prod_{i\in[n]\setminus I}p_i)\be_1,\be_2,\ldots,\be_m\rbrace$.
		In particular, we obtain $\bA\Z^{m+n-1}=\Z^m$.
		
		Let $m \le k \le m + n -1$ and let $\hat I:=\{m+n-k,\ldots,n\}$.
		By Lemma~\ref{lemma_lattice_covering} with the right hand side rounded down, it holds that 
		\begin{align}
			\label{ineq_example_2_lemma}
			\min_{\bx\in\Z^k}\|\bD_{\hat I}\bx - \bb\|_\infty \le \left\lfloor\frac{1}{2}\prod_{i=1}^{m+n-1-k} p_i\right\rfloor
		\end{align}
		for all $\bb \in \Z^m$.
		Consider the vector
		\begin{align*}
			\hat\bb := \left\lfloor\frac{1}{2}\prod_{i=1}^{m+n-1-k} p_i\right\rfloor \be_1 + N \be_2 + \cdots + N \be_m\in\Z^m = \bA\Z^{m+n-1},
		\end{align*}
		where $N$ is a sufficiently large integer.
		If $\bD$ is a submatrix of $\bA$ that does not contain all the vectors $\be_2,\ldots,\be_m$ as columns, we get $\|\bD \bx - \hat\bb\|_\infty \ge N$ for all integer $\bx$. So it suffices to consider submatrices of the form $\bD_I$ for $|I| = k - (m-1)$. For all of the corresponding sublattices $\bD_I\Z^k$, the best approximation of $\hat \bb$ is $N\be_2+\cdots+N\be_m$ as the first coordinate is a multiple of $\prod_{i\in[n]\setminus I}p_i$. This term is minimal for $\hat I$ with value $\prod_{i=1}^{m+n-1-k} p_i$. 
		Thus we have
		\[
		\min_{\substack{\bx\in\Z^{m+n-1}\\ \left|\supp(\bx)\right|\le k}}\|\bA \bx -\hat\bb\|_\infty= \min_{\substack{I\subseteq [n]\\|I|=k-(m-1)}}\min_{\bx\in\Z^{k}}\|\bD_I \bx -\hat\bb\|_\infty = \min_{\bx\in\Z^{k}}\|\bD_{\hat I} \bx -\hat\bb\|_\infty
		\]
		and the latter equals $\left\lfloor1/2\cdot\prod_{i=1}^{m+n-1-k} p_i\right\rfloor$. 
		Combining this and (\ref{ineq_example_2_lemma}) gives us
		\[
		\app_{\Z^n,k,\Vert\cdot\Vert_\infty}(\bA)=\left\lfloor\frac{1}{2}\prod_{i=1}^{m+n-1-k} p_i\right\rfloor.
		\]
		Let $p_1$ be the $k$-th prime number and $p_2,\ldots,p_n$ be the consecutive ones.
		For $k$ large, we may assume by the prime number theorem, that all numbers are of similar magnitude, say $p_1,\ldots,p_n\approx k\log_e( k) =: c$. Then we have $\delta(\bA) = \min\lbrace q_1,\ldots,q_n\rbrace \approx c^{n-1}$ and 
		\[
		\app_{\Z^n,k,\Vert\cdot\Vert_\infty}(\bA)\approx \frac{1}{2c^{k-m}} \cdot\delta(\bA),
		\]
		which matches up to the constant $c$ the exponential behavior shown in Theorem~\ref{thm_lattice_main}.
	\end{example}
	
	To capture Example~\ref{ex_lattice_exponential_behavior}, one can refine our argument in the proof of Theorem~\ref{thm_lattice_main} by choosing the $\bv_i$ that minimizes $\det\Lambda_i$ when constructing $\bD$. 
	This relates to incorporating the prime factorization of $\delta(\bA)$ similar to the techniques in the proof of \cite[Theorem 1]{alievAverkovLoeraOertel21}. 
	We omit this here for the sake of brevity and since we are mainly interested in proving the exponential behavior in $k$.
	
	Finally, using the fact that in the case that the columns of $\bA$ span the whole space positively, already a sub-selection of at most $2m$ columns do so, 
	we can derive from Theorem~\ref{thm_lattice_main} the following corollary for the case of semigroups.
	\begin{corollary}
		\label{cor_pos_spans_space}
		Let $\bA = (\ba_1,\ldots,\ba_n)\in\Z^{m\times n}$, $\pos\lbrace \ba_1,\ldots,\ba_n\rbrace  = \R^m$, and $2m\leq k\leq n$. Then there exists a $m\times k$ submatrix $\bD$ of $\bA$ such that
		\begin{align*}
			\app_{\Z^n_{\ge0},k,\Vert\cdot\Vert_\infty}(\bA) 
			\leq \max_{\bb\in \bA\Z^n}\min_{\bx\in\Z^k_{\ge0}}\Vert \bD\bx - \bb\Vert_\infty \leq \frac{1}{2^{k - 2m + 1}}\cdot \delta(\bA).
		\end{align*}
	\end{corollary}
	\begin{proof}
		Let $k'=k-m$.
		Then, by Theorem~\ref{thm_lattice_main}, there exists a $m\times k'$ submatrix $\bD'$ of $\bA$ such that for all $\bb \in \bA\Z^n$ there exists a vector $\bx\in\Z^{k'}$ with
		\begin{equation}\label{eq_proof_corollary1}
			\|\bD' \bx - \bb \|_\infty \leq \frac{1}{2^{k' - m + 1}}\cdot \delta(\bA)
		\end{equation}
		
		By the construction of $\bD'$ in the proof of Theorem~\ref{thm_lattice_main}, we may assume without loss of generality  that $\bD'$ contains the column vectors $\ba_1,\ldots,\ba_{m}$ and $\det(\ba_1,\ldots,\ba_m)=\delta(\bA)$.
		Note that $\pos\{\ba_1,\ldots,\ba_m,-\sum_{i=1}^m\ba_i\}=\R^m$.
		Since $\pos\lbrace \ba_1,\ldots,\ba_n\rbrace=\R^m$, there exists, by Carath\'eodory's Theorem, at most $m$ columns, say $\ba_{m+1},\ldots,\ba_{2m}$, such that $-\sum_{i=1}^m\ba_i\in\pos\{\ba_{m+1},\ldots,\ba_{2m}\}$.
		Extending the matrix $\bD'$ by including the column vectors $\ba_{m+1},\ldots,\ba_{2m}$, we obtain a matrix $\bD$ with at most $k$ columns.
		We claim that this matrix satisfies the statement of the corollary.
		
		Let $\bb \in \bA\Z^n$ and let $\bx\in\Z^{k'}$ satisfy \eqref{eq_proof_corollary1}.
		Assume the entry $x_\ell$ of $\bx$ is negative.
		By our above assumption we can write $x_\ell\ba_\ell=\sum_{i=1}^{2m} \mu_i\ba_{i}$, with $\mu_i\in\R_{\geq 0}$.
		All the $\mu_i$ are rational.
		Letting $L$ denote the least common multiple of all the denominators, we can substitute $x_\ell\ba_\ell$ by a positive integer combination of $\ba_{1},\ldots,\ba_{2m}$ and $\ba_\ell$ as follows.
		\[
		x_\ell\ba_\ell = \sum_{j=1}^{2m} L\mu_j\ba_{i_j} + (1-L) x_\ell\ba_\ell.
		\]
		This can be done for all negative entries of $\bx$ to express $\bD' \bx$ as a non-negative integer combination of the columns of $\bD$.
		This completes the proof.
	\end{proof}
	
	\section{General Semigroups: $m=1$}
	Let $X = \Z^{n}_{\geq 0}$ and $\bA = (\ba_1,\ldots,\ba_n)\in\Z^{m\times n}$ have full row rank. Then $\bA\cdot X$ is the semigroup generated by the columns of $\bA$. We saw that, if $\pos\lbrace \ba_1\ldots,\ba_n\rbrace = \R^m$, 
	one can apply the lattice technique from the previous section to obtain an approximation. 
	So we suppose that $\pos\lbrace \ba_1,\ldots,\ba_n\rbrace$ generates a pointed cone. 
	We begin our analysis with the case when $m = 1$ and write $\ba^\top = (a_1,\ldots,a_n)$ for the matrix $\bA$. 
	The assumption that $a_1,\ldots,a_n$ generate a pointed cone implies without loss of generality that $0<a_1\leq a_2\leq\cdots\leq a_n$. 
	Our goal is to determine the approximation error with respect to the absolute value, that is, $\app_{\ba^\top\Z^n, k , |\cdot|}(\ba)$ for a fixed $k\geq 1$. We abbreviate this parameter by $\app_k(\ba)$. 
	To bound $\app_k(\ba)$, we apply the following strategy: We fix the smallest integer $a_1$ and bound the slightly refined parameter
	\begin{align*}
		\app_{a_1,k}(\ba) := \max_{b\in \ba^\top \Z^{n}_{\geq 0}}\min \left\lbrace \left| \ba^\top\bx - b\right| : \bx\in \Z^n_{\geq 0}, \ \left|\supp(\bx)\backslash \{1\}\right|\leq k - 1\right\rbrace,
	\end{align*}
	which corresponds to the best approximation that uses $\ba_1$. By definition, we have $\app_k(\ba)\leq \app_{\ba_1,k}(\ba)$. 
	Using the pigeonhole principle one can derive an upper bound on $\app_{n-1}(\ba)$, which even holds for real valued $\ba$. 
	
	\begin{theorem}
		\label{thm_semigroup_m_1_n-1}
		Let $\ba\in\R^{n}$ with $0 < a_1\leq \cdots\leq a_n$. Then we have
		\begin{align*}
			\app_{n-1}(\ba)\leq\app_{a_1,n-1}(\ba)\leq\frac{1}{2^{n-1}}\cdot a_1.
		\end{align*}
	\end{theorem}
	Next we apply Theorem~\ref{thm_semigroup_m_1_n-1}  repeatedly to obtain a bound for general values of $k$. This bound shows that the approximation factor decreases exponentially as $k$ goes to $n$. 
	\begin{theorem}
		\label{thm_semigroup_knapsack}
		Let $\ba\in\R^{n}$ with $0 < a_1\leq \cdots\leq a_n$ and $1\leq k\leq n$. Then we have
		\begin{align*}
			\app_k(\ba) \leq \app_{a_1,k}(\ba)\leq \left(\frac{1}{2^{k-1}} - \frac{1}{2^{n-1}}\right)\cdot a_1.
		\end{align*}
	\end{theorem}
	Let us first comment on the question whether the bounds in Theorems~\ref{thm_semigroup_m_1_n-1} and  \ref{thm_semigroup_knapsack} are tight. 
	It can be shown that the bound in Theorem~\ref{thm_semigroup_m_1_n-1} is tight, even when $\ba$ is integral. More generally, we cannot get a better upper bound than $1/2^k\cdot a_1$.
	\begin{proposition}
		\label{prop_semigroup_lower_bound_m_1_n-1}
		Let $n\geq 2$ and $1\leq k\leq n - 1$. There exists $\ba\in \Z^{n}$ with $0< a_1\leq a_2\leq\cdots\leq a_n$ such that 
		\begin{align*}
			\frac{1}{2^{k}}\cdot a_1\leq \app_{k}(\ba).
		\end{align*}
	\end{proposition}
	For $k = 1$, Theorem~\ref{thm_semigroup_knapsack} can be strengthened to $\app_{a_1,1}(\ba)\leq 1/2\cdot a_1$, based on the observation that we can write every integer as an integer multiple of $a_1$ plus or minus a remainder of size at most $1/2\cdot a_1$. This construction is tight by Proposition~\ref{prop_semigroup_lower_bound_m_1_n-1} applied to $k=1$. These observations for $k\in\lbrace 1,n-1\rbrace$ suggest that $\app_k(\ba) \leq 1/2^k\cdot a_1$ might be the correct bound for all $1\leq k \leq n - 1$.  
	This is, however, not true, even when $k = 2$. 
	In fact, Theorem~\ref{thm_semigroup_approx_2} below presents a tight bound for $k = 2$. 
	To prepare for the presentation of this bound, let us  define the sequence $t_0 := 1$ and $t_i := t_{i-1}\cdot (t_{i-1} + 1)$ for $i\geq 1$. This sequence $\lbrace t_i\rbrace_{i\in \N}$ is well-known and appears in several papers related to combinatorial optimization problems; see, for instance,  \cite{galawoeginger1995on-linebinpacking,kohlikrishna1992greedyknapsack,seidenwoeginger2005twodimcuttingstock}. Moreover, the sequence $\lbrace t_i + 1\rbrace_{i\in \N}$ is known as \emph{Sylvester's sequence}. It gives among others the optimal solution to approximating $1$ by Egyptian fractions from below; see \cite{curtiss1922kelloggs,nathanson2023underapproxegyptian,oeissylvester,soundararajan2005approx_one_below} for some appearances of Sylvester's sequence in the literature. For $n\in \N$, we define 
	\begin{align*}
		\varphi(n) := \sum_{i=1}^n \frac{1}{t_i}.
	\end{align*}
	
	\begin{theorem}
		\label{thm_semigroup_approx_2}
		Let $\ba\in\R^{n}$ with $0 < a_1\leq \cdots\leq a_n$. Then we have
		\begin{align*}
			\app_2(\ba)\leq \app_{a_1,2}(\ba)\leq\frac{\varphi(n - 2)}{2\varphi(n - 2) + 1}\cdot a_1.
		\end{align*}
	\end{theorem}
	Note that for $n=3$ we get an upper bound of $1/4\cdot a_1$, for $n=4$ we get $2/7\cdot a_1$ and $29/100\cdot a_1$ for $n = 5$. In the limit, we obtain $\lim_{n\to\infty}\varphi(n) \approx 0.691$ which gives us a bound of approximately $0.2901\cdot a_1$. This improves upon the upper bound of $1/2\cdot a_1$ for $k=2$ and $n\to \infty$ given in Theorem~\ref{thm_semigroup_knapsack}. The upper bound in Theorem~\ref{thm_semigroup_approx_2} is tight.
	\begin{proposition}
		\label{prop_semigroup_lower_bound_m_1_n_2}
		Let $n\geq 2$. There exists $\ba\in \Z^{n}$ with $0<a_1\leq a_2\leq\cdots\leq a_n$ such that 
		\begin{align*}
			\frac{\varphi(n-2)}{2\varphi(n-2) + 1}\cdot a_1 \leq \app_2(\ba).
		\end{align*}
	\end{proposition}
	Our upper bounds in Theorem~\ref{thm_semigroup_m_1_n-1} and Theorem~\ref{thm_semigroup_knapsack} under the assumption that $\ba\in \Z^n$ follow from their extensions to general values of $m$ that we present in the next section. 
	We prove Theorem~\ref{thm_semigroup_approx_2} in Section~\ref{s_proof_thm_semigroup_approx_2}. The tight constructions, Propositions~\ref{prop_semigroup_lower_bound_m_1_n-1} and \ref{prop_semigroup_lower_bound_m_1_n_2}, are proven in Section~\ref{s_semigroups_lower_bounds}.
	
	\section{General Semigroups: $m\geq 2$}
	Let us now turn to the  extensions of Theorem \ref{thm_semigroup_m_1_n-1} and Theorem \ref{thm_semigroup_knapsack} to general values of  $m\geq 2$. Our strategy to extend these results is
	motivated by the analysis when $m=1$: We ensure that the cone $\pos\lbrace \ba_1,\ldots,\ba_n\rbrace$ is simplicial. To do so, we fix an invertible $m\times m$ submatrix $\bB$ of $\bA$. Without loss of generality let $\bB = (\ba_1,\ldots,\ba_m)$. Then we suppose that $\ba_i\in\pos\lbrace \ba_1,\ldots,\ba_m\rbrace$ for each $i\in\lbrack n\rbrack$. This assumption resembles the case $m=1$, where we fixed $\bB = a_1$ and also assumed that $a_i \in\pos\lbrace a_1\rbrace$. 
	We measure our approximation error in terms of the symmetric parallelepiped spanned by $\ba_1,\ldots,\ba_m$, which is defined by $P(\bB) := \bB\cdot \lbrack -1,1 \rbrack^m$. Our errors will be taken with respect to the norm $\Vert\cdot\Vert_{P(\bB)}$, the norm with unit ball $P(\bB)$. 
	This is also in correspondence with what we did when $m=1$ because, if $m=1$, we have $\bB = a_1$ and the rescaled norm induced by $1/a_1\cdot P(a_1) = \lbrack -1,1\rbrack$ equals simply the absolute value. 
	As before, let us abbreviate $\app_{\bA\Z^n_{\geq 0},k,\Vert\cdot\Vert_{P(\bB)}}(\bA)$ by $\app_k(\bA)$ for a fixed value of $k\geq m$. 
	We are now prepared to investigate the refined parameter
	\begin{align*}
		\app_{\bB,k}(\bA) := \max_{\bb\in\bA\Z^n_{\geq 0}}\min \left\lbrace \Vert \bA\bx - \bb\Vert_{P(\bB)} : \bx\in \Z^n_{\geq 0}, \ \left|\supp(\bx)\backslash \lbrack m \rbrack \right|\leq k - m\right\rbrace.
	\end{align*}
	
	\begin{theorem}
		\label{thm_semigroup_approx_n-1}
		Let $\bA = (\ba_1,\ldots,\ba_n)\in\Z^{m\times n}$ have full row rank, $m \leq n - 1$, $\bB = (\ba_1,\ldots,\ba_m)$ be invertible, $\mu = \max_{i\in\lbrack n\rbrack}\Vert\ba_i\Vert_{P(\bB)}$, and $\ba_i\in\pos\lbrace\ba_1,\ldots,\ba_m\rbrace$ for all $i\in\lbrack n\rbrack$. Then we have 
		\begin{align*}
			\app_{n-1}(\bA)\leq\app_{\bB,n-1}(\bA) \leq \frac{2}{2^\frac{n-m}{m}} \cdot \mu^{\frac{m-1}{m}}\cdot\left|\det\bB\right|^{\frac{m-1}{m}}.
		\end{align*}
	\end{theorem}
	Theorem~\ref{thm_semigroup_approx_n-1} implies  Theorem~\ref{thm_semigroup_m_1_n-1} when $m=1$ and $\ba \in \Z^n$.
	As for $m=1$, Theorem~\ref{thm_semigroup_approx_n-1} applied repeatedly allows us to  observe that the approximation factor decreases exponentially  as $k$ goes to $n$.
	\begin{theorem} 
		\label{thm_semigroup_general_bound}
		Let $\bA = (\ba_1,\ldots,\ba_n)\in\Z^{m\times n}$ have full row rank. Let $\bB = (\ba_1,\ldots,\ba_m)$ be invertible, $\mu = \max_{i\in\lbrack n\rbrack}\Vert\ba_i\Vert_{P(\bB)}$, and $\ba_i\in\pos\lbrace\ba_1,\ldots,\ba_m\rbrace$ for all $i\in\lbrack n\rbrack$. Then we have
		\begin{align*}
			\app_k(\bA)\leq \app_{\bB,k}(\bA) \leq \frac{{2}}{2^{\frac{1}{m}} - 1}\cdot \left(\frac{1}{2^{\frac{k-m}{m}}} - \frac{1}{2^{\frac{n-m}{m}}}\right)\cdot \mu^{\frac{m-1}{m}}\cdot \left|\det\bB\right|^{\frac{m-1}{m}}
		\end{align*}
		for all $m\leq k\leq n$.
	\end{theorem}
	By evaluating the inequality in Theorem~\ref{thm_semigroup_general_bound} at $m = 1$ and rescaling the norm, we recover the bound given in Theorem~\ref{thm_semigroup_knapsack} provided that $\ba \in \Z^n$. We prove Theorem~\ref{thm_semigroup_approx_n-1} and Theorem~\ref{thm_semigroup_general_bound} in Section~\ref{s_proof_thm_semigroup_approx_n-1}. 
	
	A few comments on the bound in Theorem~\ref{thm_semigroup_general_bound} are in order. The bounds for $m\geq 2$ depend on the factor $\mu$ whereas the bounds for $m = 1$ are independent of $\mu$. This is not a coincidence. 
	In fact, for $m\geq 2$, it is possible to show that a dependence on $\mu$ (or some related parameter  measuring the size of $\bA$) is necessary to obtain approximation factors that decrease in $k$. This follows from the result below that we prove in Section~\ref{s_semigroups_lower_bounds}.
	
	\begin{proposition}
		\label{prop_semigroup_lower_bound_general}
		Let $n\geq 3$. There exists $\bA = (\ba_1,\ldots,\ba_n)\in\Z^{2\times n}$ such that  $\ba_i\in\pos\lbrace \ba_1,\ba_2\rbrace$ for all $i\in\lbrack n\rbrack$ and
		\begin{align*}
			\frac{\left\lfloor \sqrt{n - 1}\right\rfloor}{n-1}\leq \app_{k}(\bA)
		\end{align*}
		for all $k\in \lbrace 2,\ldots,n-1\rbrace$.
	\end{proposition}
	
	Let us continue with a discussion of the bound established in Theorem~\ref{thm_semigroup_general_bound}. Since we assume that $\ba_i\in\pos\lbrace \ba_1,\ldots,\ba_m\rbrace$ for all $i\in\lbrack n\rbrack$, we always have an upper bound $\app_k(\bA)\leq \app_{\bB,k}(\bA) \leq 1/2$ whenever $k\geq m$. This follows by only using the vectors $\ba_1,\ldots,\ba_m$. For small values of $k$, however, the bound in Theorem~\ref{thm_semigroup_general_bound} is worse than $1/2$. In particular, one can calculate that the bound in Theorem~\ref{thm_semigroup_general_bound} improves upon $1/2$ if we have $k\geq 2m+m\log_2(cm\mu^{(m-1)/m}\left|\det\bB\right|^{(m-1)/m})$ for some constant $c>0$. These calculations utilize the fact that ${2}/(2^{1/m} - 1) = \mathcal{O}(m)$ for $m\geq 1$. 
	On the other hand, the results in \cite{alievdeloesparelindio2017} and \cite{eisenbrandshmonincaratheodorybounds06} establish that there exists an exact representation whenever 
	\begin{align*}
		k & \geq 2m\log_2(2\sqrt{m}\Vert\bA\Vert_\infty) & \cite{alievdeloesparelindio2017},\\
		k & \geq 2m\log_2(4m\Vert\bA\Vert_\infty) & \cite{eisenbrandshmonincaratheodorybounds06}.
	\end{align*}
	These results are stated in terms of $\Vert \bA\Vert_\infty$. The bound in Theorem~\ref{thm_semigroup_general_bound} is based on the parameters $\mu$ and $\left| \det\bB \right|$. The proof of the inequality in \cite{alievdeloesparelindio2017} is based on Siegel's lemma. We do not see how to phrase this in terms of $\mu$ and $\left|\det\bB\right|$. However, the proof of the inequality in \cite{eisenbrandshmonincaratheodorybounds06} can be translated to give a bound on an exact representation of the form $k \geq 2m\log_2(4m\mu\left|\det\bB\right|^{1/m})$. Hence, the bound provided in Theorem~\ref{thm_semigroup_general_bound} is relevant whenever
	\begin{align*}
		2m+m\log_2(cm\mu^{\frac{m-1}{m}}\left|\det\bB\right|^{\frac{m-1}{m}})\leq k\leq 2m\log_2(4m\mu\left|\det\bB\right|^{\frac{1}{m}}).
	\end{align*}
	This is the case when $\mu$ is relatively large compared to $m$ and $\left|\det\bB\right|$. In this regime, the bounds in \cite{alievdeloesparelindio2017} and \cite{eisenbrandshmonincaratheodorybounds06} are equal up to a constant. We also obtain a new bound on the size of an exact representation.
	\begin{corollary}
		\label{cor_semigroup_approx_to_exact}
		Let $\bA = (\ba_1,\ldots,\ba_n)\in\Z^{m\times n}$ have full row rank. Let $\bB = (\ba_1,\ldots,\ba_m)$ be invertible, $\mu = \max_{i\in\lbrack n\rbrack}\Vert\ba_i\Vert_{P(\bB)}$, and $\ba_i\in \pos\lbrace \ba_1,\ldots,\ba_n\rbrace$ for all $i\in\lbrack n\rbrack$. Then we have $\app_{n-1}(\bA) = 0$ for 
		\begin{align*}
			n\geq m + m\log_2({2}\mu^{(m-1)/m}\left|\det\bB\right|^{(2m-1)/m}).
		\end{align*}
	\end{corollary}
	\begin{proof}
		The polytope $1/\left|\det\bB\right|\cdot P(\bB)$ does not contain non-zero integer vectors in its interior since each non-zero integer vector can be expressed in terms of $\bB$ and non-zero coefficients which are at least $1/\left|\det\bB\right|$. So we have
		\begin{align*}
			1/\left|\det\bB\right|\cdot \intt(P(\bB))\cap\Z^n = \lbrace\bm{0}\rbrace. 
		\end{align*}
		If ${2\cdot}\mu^{(m-1)/m}\cdot\left|\det\bB\right|^{(2m-1)/m} < 2^{(n-m)/m}$, Theorem~\ref{thm_semigroup_approx_n-1} implies $\app_{\bB,n-1}(\bA) < 1/\left|\det\bB\right|$. 
		This means that we can find an approximation of some $\bb\in\bA\Z^n_{\geq 0}$ that is contained in $\bb + 1/\left|\det\bB\right|\cdot\intt( P(\bB))$.
		However, we have 
		\begin{align*}
			\left(\bb + 1/\left|\det\bB\right|\cdot\intt( P(\bB)\right) \cap \Z^n = \lbrace \bb\rbrace.	
		\end{align*}
		Hence, $\bb$ can only be approximated by $\bb$ itself, which means that $\bb$ can be expressed as non-negative integer combination of at most $n-1$ columns and thus gives an approximation error of $0$.
		Rearranging terms to isolate $n$ in ${2\cdot}\mu^{(m-1)/m}\cdot\left|\det\bB\right|^{(2m-1)/m} < 2^{(n-m)/m}$ gives the claimed bound.
	\end{proof}
	When $m=1$, we recover a tight bound on the sparse representation of integers in semigroups; cf. \cite[Theorem 3]{alievAverkovLoeraOertel21}.  
	
	Our results are based on the assumption that $\ba_i\in\pos\lbrace \ba_1,\ldots,\ba_m\rbrace$ for all $i\in\lbrack n\rbrack$ or, equivalently, $\pos\lbrace \ba_1,\ldots,\ba_n\rbrace$ is simplicial. This assumption allows us to fix a suitable basis, denoted above by $\bB$, which we use consistently throughout our arguments. In particular, we bound the refined parameter $\app_{\bB,k}(\bA)$ and measure the approximation error with a norm that depends on $\bB$. 
	In general, if one removes the assumption that $\ba_i\in\pos\lbrace \ba_1,\ldots,\ba_m\rbrace$ for all $i\in\lbrack n\rbrack$, it is not clear how to obtain results in the flavor of Theorems~\ref{thm_semigroup_approx_n-1} and \ref{thm_semigroup_general_bound}. If such an extension is possible, it likely requires new approaches and further assumptions. The following example supports this statement by illustrating that for a given basis $\bB$ certain vectors can be hard to approximate.
	
	\begin{example}
		Let $l\geq 3$ and define
		\begin{align*}
			\bA := (\ba_1,\ba_2,\ba_3,\ba_4) =
			\begin{pmatrix}
				1 & 0 & 2^{l-1} & 2^{l-1} + 2^{l-2}\\ 0 & 1 & -2^{l-1} & -2^l
			\end{pmatrix}.
		\end{align*}
		Choose $\bB := (\ba_1,\ba_2)$, note that $\ba_3,\ba_4\notin\pos\lbrace \ba_1,\ba_2\rbrace$. Select
		\begin{align*}
			\bb := \ba_3 + \ba_4 = \begin{pmatrix}
				2^l + 2^{l-2} \\ -2^l - 2^{l-1}
			\end{pmatrix}.
		\end{align*}
		Then one can calculate
		\begin{align*}
			2^{\frac{l - 4}{2}}\cdot\mu^{\frac{1}{2}}\cdot\left|\det\bB\right|^{\frac{1}{2}} = 2^{l-2} \leq \app_{\bB,3}(\bA),
		\end{align*}
		which shows that, in contrast to Theorems~\ref{thm_semigroup_approx_n-1} and \ref{thm_semigroup_general_bound}, no upper bounds on $\app_{\bB,3}(\bA)$ in terms of $\mu^{1/2}$ and $\left|\det\bB\right|^{1/2}$ exist.
	\end{example}
	To put the lower bound of $2^{l-2}$ in the example into context, note that for any $\bb \in\bA\Z^n_{\geq 0}$, one can select a basis $\bB$ such that $\bb$ is contained in the cone spanned by the columns of $\bB$. This guarantees $\app_{\bB,k}(\bA)\leq 1/2$ for all $k\geq m$. However, the choice of the basis $\bB$ differs with varying $\bb$.
	
	\section{Proofs of Theorem~\ref{thm_semigroup_approx_n-1} and Theorem~\ref{thm_semigroup_general_bound}}
	\label{s_proof_thm_semigroup_approx_n-1}
	The majority of our effort is dedicated to proving Theorem~\ref{thm_semigroup_approx_n-1}, from which Theorem~\ref{thm_semigroup_general_bound} will follow as a consequence.
	
	Let us start to describe the major steps required to prove Theorem~\ref{thm_semigroup_approx_n-1}. Let $\bA$ be given as in Theorem~\ref{thm_semigroup_approx_n-1}. First, we show that it suffices to approximate only a finite set of integer vectors with small size. For that purpose, fix $\bb\in\bA\Z^n_{\geq 0}$ and $k\geq m$.  Then we define the approximation error of $\bb$ with respect to $\bB$ and $k$ to be
	\begin{align*}
		\app_{\bB,k}(\bA,\bb) := \min \left\lbrace \Vert \bA\bx - \bb\Vert_{P(\bB)} : \bx\in \Z^n_{\geq 0}, \ \left|\supp(\bx)\backslash \lbrack m \rbrack \right|\leq k - m\right\rbrace
	\end{align*}
	and the following finite set 
	\begin{align*}
		S:=\left\lbrace \bb = \lambda_{m+1}\ba_{m+1} + \cdots + \lambda_n\ba_n : \lambda_{m+1},\ldots,\lambda_n\in \N, \sum_{i= m + 1}^n\lambda_{i}\leq \left|\det\bB\right|\right\rbrace.
	\end{align*}
	The first result states that it suffices to approximate the finitely many integer vectors that lie in the set $S$.
	\begin{lemma}
		\label{lemma_general_semigroup_small_vectors}
		Let $\bA = (\ba_1,\ldots,\ba_n)\in\Z^{m\times n}$ have full row rank, $m\leq k\leq n - 1$, $\bB = (\ba_1,\ldots,\ba_m)$ be invertible, $\mu = \max_{i\in\lbrack n\rbrack}\Vert\ba_i\Vert_{P(\bB)}$, and $\ba_i\in\pos\lbrace\ba_1,\ldots,\ba_m\rbrace$ for all $i\in\lbrack n\rbrack$. Then we have 
		\begin{align*}
			\app_{\bB,k}(\bA) = \max_{\bb\in S}\app_{\bB,k}(\bA,\bb)
		\end{align*}
	\end{lemma}
	\begin{proof}
		We will show that for all $\bb\notin S$ one can construct a vector $\bb'\in S$ with $\app_{\bB,k}(\bA,\bb) = \app_{\bB,k}(\bA,\bb')$, which implies the claim. 
		
		Let $\bb\notin S$ with $\bb = \lambda_1\ba_1 + \cdots + \lambda_n\ba_n$. Since we fix $\ba_1,\ldots,\ba_m$ in our non-negative integral combination, we assume without loss of generality that $\lambda_i = 0$ for all $i\in\lbrack m\rbrack$. We analyze the sum
		\begin{align*}
			\bb = \underbrace{\ba_{m+1} + \cdots + \ba_{m + 1}}_{\lambda_{m+1}\text{ times }} + \cdots + \underbrace{\ba_n + \cdots + \ba_n}_{\lambda_n\text{ times}}.
		\end{align*}
		Let $\by_j$ denote the subsum of the first $j$ terms on the right from above for $j\in \lbrack \sum_{i=m+1}^n \lambda_i\rbrack$. There are at least $\left|\det\bB\right| + 1$ such subsums since $\bb\notin S$. Moreover, there exist  $\left|\det\bB\right|$ different affine lattices of the form $\by + \bB\Z^m$ for $\by\in \Z^m$. 
		By the pigeonhole principle, we can find two different subsums $\by_i$ and $\by_j$ with $i < j$ such that $\by_j - \by_i\in \bB\Z^m$. Furthermore, we have $\by_j - \by_i\in\pos\lbrace \ba_1,\ldots,\ba_m\rbrace$ as $i<j$ and $\ba_i\in\pos\lbrace \ba_1,\ldots,\ba_m\rbrace$ for all $i\in\lbrack n\rbrack$. So we have $\by_j - \by_i = \sum_{i=1}^m\mu_i\ba_i$ for non-negative integral values $\mu_1,\ldots,\mu_m$. Instead of approximating $\bb$, we approximate the new integer vector $\bb- (\by_j - \by_i)\in \pos\lbrace \ba_1,\ldots,\ba_m\rbrace$. Once we have an approximation of this new vector, we add the vector $\sum_{i=1}^m\mu_i\ba_i$ back and obtain an approximation of $\bb$ with the same error. Since $\bb- (\by_j - \by_i)$ has a strictly smaller coordinate sum than $\bb$, one can repeat these steps until the coordinate sum is bounded by $\left|\det\bB\right|$.
	\end{proof}
	The second part of the proof of Theorem~\ref{thm_semigroup_approx_n-1} follows a similar approach to the techniques in \cite{eisenbrandshmonincaratheodorybounds06}: 
	We fix some $\bb\in S$. This vector has a representation of the form $\bb = \lambda_{m+1}\ba_{m+1} + \cdots + \lambda_n\ba_n$ with $\lambda_{m+1},\ldots,\lambda_n\in\N$.  
	We tile $\pos\lbrace \ba_1,\ldots,\ba_m\rbrace$ with sufficiently small copies of the parallelepiped spanned by $\ba_1,\ldots,\ba_m$. Then we count how many subsums of $\bb = \lambda_{m+1}\ba_{m+1} + \cdots + \lambda_n\ba_n$ are contained in the copies of the sufficiently small parallelepipeds. 
	It turns out that the number of subsums grows exponentially in $n$ and the number of parallelepipeds that we need to consider only grows exponentially in $m$. Once this has been established one could derive a first bound that is, however, weaker than the one stated in Theorem~\ref{thm_semigroup_approx_n-1}. (The key distinction is that one would replace the expression $\mu^{(m-1)/m}\cdot\left|\det\bB\right|^{(m-1)/m}$ by the expression $\mu\cdot \left|\det\bB\right|$.) 
	To get Theorem~\ref{thm_semigroup_approx_n-1}, a second ingredient must be incorporated, one that, to the best of our knowledge, has not previously been utilized: the idea is to show that we only need to consider so-called incomparable subsums. 
	
	To make this notion precise, we work with partially ordered sets. We refer the reader to \cite[Chapter 3]{StanleyEC1} for definitions. Let  $(P_{m,s},\preceq)$ be the partially order set given by $P_{m,s} := \Z^m \cap \lbrack 0, s\rbrack^m$ and $\by \preceq \bz$ if and only if $y_1\leq z_1,\ldots,y_m\leq z_m$ for $m,s\geq 1$. An antichain of $(P_{m,s},\preceq)$ is a subset of $P_{m,s}$ whose elements are pairwise incomparable. An upper bound on the size of a largest antichain in $(P_{m,s},\preceq)$ can be established easily. 
	
	\begin{lemma}
		\label{lemma_semigroup_partial_order}
		Let $m,s\in\N_{\geq 1}$. The largest antichain of $(P_{m,s},\preceq)$ is bounded above by $(s+1)^{m-1}$. It is in the order  of $\Theta(s^{m-1})$.
	\end{lemma}
	\begin{proof}
		We  determine the upper bound first. By Dilworth's theorem \cite{dilworth1950decomposition}, the largest antichain equals the minimal size of a chain cover, a union of chains that cover the partially ordered set. Thus, we can construct a chain cover to get an upper bound on the size of the largest antichain. Define the chains $C_{z_1,\ldots,z_{m-1}} := \lbrace (z_1,\ldots,z_{m-1},0)^\top,\ldots, (z_1,\ldots,z_{m-1},s)^\top\rbrace$ for fixed values $z_1,\ldots,z_{m-1}\in\lbrace 0,\ldots,s\rbrace$. Those chains cover $P_{m,s}$. In total, there are $(s+1)^{m-1}$ chains of the type $C_{z_1,\ldots,z_{m-1}}$. This gives the desired upper bound. 
		
		For the asymptotic bound, note that the upper bound  implies that the size of a largest antichain is in $\mathcal{O}(s^{m-1})$. Hence, it suffices to construct an antichain of size $\Omega(s^{m-1})$. Define the subset of $P_{m,s}$ given by $S := \lbrace \bz\in P_{m,s} : \sum_{i=1}^m z_i = s\rbrace$. The elements in $S$ are incomparable by construction. From Ehrhart theory, cf. \cite[Chapter 3]{beck2007computing}, it follows that 
		$\left|S\right| = \Omega(s^{m-1})$ since the elements in $S$ lie on an $(n-1)$-dimensional affine hyperplane.
	\end{proof}
	Let $\parr\lbrace \ba_1,\ldots,\ba_m\rbrace$ denote the half-open parallelepiped spanned by the basis vectors $\ba_1,\ldots,\ba_m$. Taking  the parameter $\mu$ into consideration one has that  $$\ba_i\in\mu\cdot \parr\lbrace \ba_1,\ldots,\ba_m\rbrace \text{ for all } i\in\lbrack n\rbrack.$$
	
	\begin{proof}[Proof of Theorem~\ref{thm_semigroup_approx_n-1}]
		Suppose that $\bb \in S$ by Lemma~\ref{lemma_general_semigroup_small_vectors}. The definition of $S$ guarantees the existence of $\lambda_{m+1},\ldots,\lambda_n\in \N$ such that $\bb = \lambda_{m+1}\ba_{m+1} + \cdots + \lambda_n\ba_n$ and $\sum_{i= m + 1}^n\lambda_{i}\leq \left|\det\bB\right|$. 
		Using the latter property and the triangle inequality we obtain $\bb \in \mu \cdot \left|\det\bB\right|\cdot\parr\lbrace \ba_1,\ldots,\ba_m\rbrace$ as $\ba_i \in \mu\cdot\parr\lbrace \ba_1,\ldots,\ba_m\rbrace$ for all $i\in\lbrack n\rbrack$.
		We tile $\mu \cdot \left|\det\bB\right|\cdot\parr\lbrace \ba_1,\ldots,\ba_m\rbrace$ using $\mu^m \cdot\left|\det\bB\right|^m$ copies of the parallelepiped $\parr\lbrace \ba_1,\ldots,\ba_m\rbrace$. 
		We refine this tiling by further dividing each $\parr\lbrace \ba_1,\ldots,\ba_m\rbrace$ into $l^m$ parellelepipeds of size and shape $1/l\cdot\parr\lbrace \ba_1,\ldots,\ba_m\rbrace$. In total, we have $\mu^m\cdot \left|\det\bB\right|^m \cdot l^m$ parallelepipeds tiling $\mu \cdot \left|\det\bB\right|\cdot\parr\lbrace \ba_1,\ldots,\ba_m\rbrace$. 
		Let $\tilde{\ba}_i := \lambda_i\ba_i$. Suppose that two subsums of $\tilde{\ba}_{m + 1} + \cdots + \tilde{\ba}_n$ are in the same tile, up to adding multiples of $\ba_1,\ldots,\ba_m$ to one of the subsums. In other words, suppose that we have two different sets $I,J\subseteq \lbrace m + 1,\ldots, n\rbrace$ such that $\mu_1\ba_1 + \cdots + \mu_m\ba_m + \sum_{i\in I}\tilde{\ba}_i$ and $\sum_{j\in J}\tilde{\ba}_j$ are in the same tile for $\mu_1,\ldots,\mu_m\in \N$. This implies $\Vert \mu_1\ba_1 + \cdots + \mu_m\ba_m + \sum_{i\in I}\tilde{\ba}_i - \sum_{j\in J}\tilde{\ba}_j\Vert_{P(\bB)}\leq 1/l$. Since $I,J$ are different sets, we assume without loss of generality that there exists $j^*\in J$ with $j^*\notin I$. If this is not the case, we have $J\subsetneq I$, which implies $\Vert\sum_{j\in J}\tilde{\ba}_j\Vert_{P(\bB)} < \Vert\sum_{i\in I}\tilde{\ba}_i$$\Vert_{P(\bB)}$. So we have $\mu_i = 0$ for all $i\in\lbrack m\rbrack$ and we can interchange the role of $I$ and $J$. 
		Using this, we construct a new approximate vector 
		\begin{align*}
			\bar{\bb} := \bb + \mu_1\ba_1 + \cdots + \mu_m\ba_m + \sum_{i\in I}\tilde{\ba}_i - \sum_{j\in J}\tilde{\ba}_j.
		\end{align*}
		This is an element of the semigroup that is generated by $n - 1$ elements since we do not use the generator $\tilde{\ba}_{j^*}$. We get  
		\begin{align}
			\label{semigroup_inequality_n-1}
			\Vert \bar{\bb} - \bb\Vert_{P(\bB)} = \Vert \mu_1\ba_1 + \cdots + \mu_m\ba_m + \sum_{i\in I}\tilde{\ba}_i - \sum_{j\in J}\tilde{\ba}_j\Vert_{P(\bB)}\leq \frac{1}{l}.
		\end{align}
		It remains to find a bound on $l$. To accomplish this task, we investigate how fine we want to tile $\parr\lbrace \ba_1,\ldots,\ba_m\rbrace$ while simultaneously ensuring that two subsums of $\tilde{\ba}_{m+1} + \cdots + \tilde{\ba}_n$ are in the same tile, up to adding multiples of $\ba_1,\ldots,\ba_m$ to one of the subsums. 
		We have $l^m$ tiles in $\parr\lbrace \ba_1,\ldots,\ba_m\rbrace$. 
		Fix one of those tiles and call it $T$.  
		We define $T(\mu_1,\ldots,\mu_m):= T + \sum_{i=1}^m\mu_i\ba_i$ for $\mu_1,\ldots,\mu_m\in \lbrace 0,\ldots,\mu\cdot\left|\det\bB\right| - 1\rbrace$. These are the translates of $T$ modulo $\ba_1,\ldots,\ba_m$ that are contained in $\mu\cdot\left|\det\bB\right|\cdot\parr\lbrace \ba_1,\ldots,\ba_m\rbrace$. 
		There are $\mu^m\cdot\left|\det\bB\right|^m$ of those translates. We claim that if there are more than $\mu^{m-1}\cdot\left|\det\bB\right|^{m-1}$ subsums in different $T(\mu_1,\ldots,\mu_m)$, we can add multiples of $\ba_1,\ldots, \ba_m$ to one of the subsums and end up in the same translated tile. To see this, we define a partial order on the set $\mathcal{T}$ of the translates of $T(\mu_1,\ldots,\mu_m)$ for $\mu_1,\ldots,\mu_m\in\lbrace 0,\ldots,\mu\cdot\lambda - 1\rbrace$. Let $T(\mu_1,\ldots,\mu_m)\preceq T(\mu'_1,\ldots,\mu'_m)$ if and only if $\mu_1\leq \mu'_1,\ldots,\mu_m\leq\mu'_m$. This partial order implies that if two subsums are contained in translated tiles that obey the ordering above, we are done since $T(\mu_1,\ldots,\mu_m) + \sum_{i=1}^m(\mu'_i - \mu_i) = T(\mu'_1,\ldots,\mu'_m)$ and $\mu'_i - \mu_i\geq 0$ for all $i\in\lbrack m\rbrack$. Hence, we want to find the largest subset of $\mathcal{T}$ whose elements are pairwise incomparable with respect to the partial order. This is equivalent to determining the largest antichain in $(\mathcal{T},\preceq)$. Observe that $(\mathcal{T},\preceq)$ is isomorphic to $(P_{m,\mu\cdot\left|\det\bB\right| - 1},\preceq)$. So we can apply Lemma~\ref{lemma_semigroup_partial_order} to obtain the upper bound $\mu^{m-1}\cdot\left|\det\bB\right|^{m-1}$ on the size of the largest antichain in $(\mathcal{T},\preceq)$. 
		This holds for all $l^m$ tiles. 
		Hence, by the pigeonhole principle we must ensure that 
		\begin{align*}
			2^{n-m}\geq \mu^{m-1}\cdot \left|\det\bB\right|^{m-1}\cdot l^m
		\end{align*}
		to guarantee that two subsums are, up to adding multiples of $\ba_1,\ldots,\ba_m$, in the same tile. 
		To get the best possible bound in \eqref{semigroup_inequality_n-1}, we have to choose $l$ as large as possible such that the inequality above holds. Solving for $l$ yields
		\begin{align*}
			\frac{2^{\frac{n - m}{m}}}{\mu^{\frac{m-1}{m}}\cdot \left|\det\bB\right|^{\frac{m-1}{m}}}\geq l.
		\end{align*}
		In addition to satisfying this inequality, $l$ must be an integer. One could simply take the left hand side rounded down. However, to not complicate the presentation of the result, we instead divide the left hand side by two and set $l := 2^{-1}\cdot 2^{(n-m)/m}\cdot (\mu\cdot\left|\det\bB\right|)^{-(m-1)/m}$. This choice is feasible, because we may assume that the left hand side of the inequality is at least one, otherwise, the bound in Theorem~\ref{thm_semigroup_approx_n-1} would hold trivially as $\app_{\bB,k}(\bA) \leq 1/2$ for all $k\geq m$. So dividing by two guarantees the existence of an integer that is at least as large as our chosen value $l$. Substituting $l$ into (\ref{semigroup_inequality_n-1}) proves the claim. 
	\end{proof}
	
	The lower bound in Lemma~\ref{lemma_semigroup_partial_order} implies that we cannot asymptotically improve upon the bound $\mu^{m-1}\cdot\left|\det\bB\right|^{m-1}$ when estimating the size of a largest antichain.
	
	\begin{proof}[Proof of Theorem~\ref{thm_semigroup_general_bound}]
		The claim follows immediately when $k = n$. So suppose that $k\leq n - 1$. Let ${\bb} = \bA {\bx}$ for ${\bx}\in\Z^n_{\geq 0}$ be given. We set $\bx^0 := \bx$ and $I_{0} := \supp(\bx)\backslash \lbrack m\rbrack$.
		For $j \geq 1$, we define $I_j$ and $\bx^j$ recursively as follows:
		\begin{enumerate}
			\item Let $\bA_{I_{j-1}} $ be the matrix with  columns  in $\lbrack m\rbrack\cup I_{j-1}$ and  $\bb^{j-1} = \bA {\bx}^{j-1}$.
			\item Apply Theorem~\ref{thm_semigroup_approx_n-1} to $\bb^{j-1}$ and $\bA_{I_{j-1}}$. We obtain an approximation $ {\bb}^j = \bA\bx^j$ for some $\bx^j\in\Z^{n}_{\geq 0}$ with $\supp(\bx^j)\subseteq \lbrack m\rbrack \cup I_{j-1}$ and $\left|\supp(\bx^j)\backslash\lbrack m\rbrack\right| < |I_{j-1}|$. 
			\item If $\left|\supp(\bx^j)\backslash\lbrack m\rbrack\right| \leq k - m$, we are done. 
			Otherwise, define $I_j := \supp(\bx^j)\backslash\lbrack m\rbrack$ and return to the first step.
		\end{enumerate}
		Since we reduce $|I_j|$ by at least one in each step, this procedure terminates with some $\bx^l$ such that $\left|\supp(\bx^l)\backslash\lbrack m\rbrack\right|\leq k - m$ and $1\leq l\leq n-k$. We observe
		\begin{align}
			\label{inequ_semigroup_approx_n-1}
			\app_k(\bA)\leq\app_{\bB,k}(\bA)\leq \sum_{j = 0}^{l-1} \app_{\bB,\left|I_j\right| + m - 1}(\bA_{I_j}).
		\end{align}
		It remains to bound $\app_{\bB,\left|I_j\right| + m - 1}(\bA_{I_j})$ using Theorem~\ref{thm_semigroup_approx_n-1}. To this end let 
		$\tau := 2\cdot \mu^{(m-1)/m}\left|\det\bB\right|^{(m-1)/m}.$
		Inequality (\ref{inequ_semigroup_approx_n-1}) and Theorem~\ref{thm_semigroup_approx_n-1} give 
		\begin{align*}
			\app_{k}(\bA)\leq \tau\cdot\sum_{j = 0}^{l-1}\left(\frac{1}{2^{\left| I_j\right|}}\right)^\frac{1}{m}\leq \tau\cdot\sum_{j = 0}^{n - k - 1}\left(\frac{1}{2^{n - m - j}}\right)^\frac{1}{m}
		\end{align*}
		since $l\leq n - k$.  With the geometric sum formula one obtains
		\begin{align*}
			\sum_{j = 0}^{n - k - 1}\left(\frac{1}{2^{n - m - j}}\right)^\frac{1}{m} = \frac{1}{2^{\frac{n - m}{m}}}\cdot\sum_{j = 0}^{n- k - 1} 2^{\frac{j}{m}} = \frac{1}{2^{\frac{1}{m}} - 1}\cdot \left(\frac{1}{2^{\frac{k-m}{m}}} - \frac{1}{2^{\frac{n-m}{m}}}\right). 
		\end{align*}
	\end{proof}

	\section{Proof of Theorem~\ref{thm_semigroup_approx_2}}
	\label{s_proof_thm_semigroup_approx_2}
	Let $n\geq 3$. We abbreviate the bound in Theorem~\ref{thm_semigroup_approx_2} by 
	$$\epsilon := \frac{\varphi(n - 2)}{2\varphi(n - 2) + 1}\cdot a_1.$$  
	We first argue that without loss of generality we can suppose that $b := a_2+\cdots+ a_n$: 		
	If $b = \lambda_1 a_1 + \cdots + \lambda_n a_n$ for $\lambda_i\in \N$, we can assume that $\lambda_i > 0$ for all $i\geq 2$ since, otherwise, we pass to a setting with fewer integers. 
	Similarly, if $\lambda_i > 1$ for  an index $i\geq 2$, replace the integer $a_i$ by $\tilde{a}_i :=\lambda_i a_i$. 
	Since we fix $a_1$ in our non-negative combination, we also assume that $\lambda_1 = 0$.
	
	Let $a_i = z_ia_1 +  r_i$ with $r_i\in  (- a_1 / 2,a_1/2\rbrack$ and $z_i\in\N $ for $i\in\lbrace 2,\ldots,n\rbrace$.  Our target is to approximate $a_2 + \cdots + a_n$ using $a_1$ and at most one additional integer $a_i$. To do so, we work with the residues $r_i$ and approximate the sum of residues $r_2 + \cdots + r_n$ using only one $r_i$. Without loss of generality we can assume that $r_2 + \cdots + r_n > 0$. Define $I := \lbrace i\in\lbrace 2,\ldots,n\rbrace : r_i > 0\rbrace$ and, hence, $r_2 + \cdots + r_n \leq \sum_{i\in I}r_i$. We define $\tilde{s}_i := \lfloor \sum_{j\in I} a_j / a_i \rfloor$ for all $i\in I$. 
	Consider the following integer knapsack problem:
	\begin{equation}\label{intknap}
		\max \sum_{i\in I}r_ix_i \ \text{ s.t. } \ \sum_{i\in I}\frac{1}{\tilde{s}_i + 1}x_i < 1, \ x_i\in\Z_{\geq 0} \text{ for all }i\in I.
	\end{equation}
	The knapsack constraint is chosen such that the all-ones vector is a feasible solution to the problem. This is true since
	\begin{align*}
		\sum_{i\in I}\frac{1}{\tilde{s}_{i} + 1} < \sum_{i\in I}\frac{a_i}{\sum_{j\in I}a_j} = 1.
	\end{align*}
	Let $z_1^*,\ldots,z^*_n$ be an optimal solution for \eqref{intknap} and $\tilde{s}_{i^*}r_{i^*} = \max_{i\in I} \tilde{s}_i r_i$. From Theorem~\ref{thm_kk_variant} from the appendix, which is a variant of a result by Kohli and Krishnamurti \cite[Theorem 2]{kohlikrishna1992greedyknapsack},
	one obtains
	\begin{align*}
		\sum_{i\in I}r_iz^*_i\leq \left(1 + \varphi(|I| - 1)\right) \tilde{s}_{i^*}r_{i^*}. 
	\end{align*}
	Note that to get this inequality, one has to slightly adjust the proof of \cite[Theorem 2]{kohlikrishna1992greedyknapsack}.
	This is shown in detail in the appendix.
	Then, since the all-ones vector is a feasible integral solution to \eqref{intknap}, we conclude that
	\begin{align}
		\label{inequ_semigroup_proof_k=2}
		\sum_{i\in I}r_i - \tilde{s}_{i^*}r_{i^*}\leq \sum_{i\in I} r_iz^*_i - \tilde{s}_{i^*}r_{i^*}\leq \varphi(|I| - 1) \tilde{s}_{i^*}r_{i^*}.
	\end{align}
	We next derive an upper bound on $\tilde{s}_{i^*}r_{i^*}$. 
	
	Let $s_{i^*} := \lfloor \sum_{j=2}^na_j/a_{i^*}\rfloor$ and denote by $\bar{b}$ the least non-negative value in $r_2 + \cdots + r_n + a_1\Z = b + a_1\Z$. Suppose next that there exists an integer $k\in \lbrace 0,\ldots,s_{i^*}\rbrace$ such that $kr_{i^*}\in \lbrack \bar{b} - \epsilon, \bar{b} + \epsilon\rbrack$. This implies that there exists an integer  $\lambda $ such that $|\lambda a_1 + k a_{i^*} - b| = |\bar{b} - kr_{i^*}|\leq \epsilon$. The integer $\lambda$ is non-negative because $k a_{i^*}\leq a_2+\cdots+ a_n = b$ by  definition of $s_{i^*}$. Hence, the theorem is correct under this assumption. 
	
	From now on we assume that $kr_{i^*}\notin \lbrack \bar{b} - \epsilon, \bar{b} + \epsilon\rbrack$ for all $k\in \lbrace 0,\ldots,s_{i^*}\rbrace$.
	We claim that we can strengthen this assumption to $kr_{i^*} < \bar{b} - \epsilon$ for all $k\in\lbrace 0,\ldots,s_{i^*}\rbrace$. Suppose this is not true. Let $k^*\in\lbrace 0,\ldots,s_{i^*}\rbrace$ be the smallest index such that $k^*r_{i^*} > \bar{b} + \epsilon$. As $0 < \bar{b} + \epsilon$, it follows that $k^*\geq 1$. We obtain
	\begin{align*}
		r_{i^*} = k^*r_{i^*} - (k^* - 1)r_{i^*} > \bar{b} + \epsilon - (\bar{b} -\epsilon) = 2\epsilon \geq \frac{1}{2} a_1,
	\end{align*}
	where we use $\epsilon\geq 1/4 a_1$ and $(k^* - 1)r_{i^*}<\bar{b}-\epsilon$ due to the minimality of $k^*$ and $(k^* - 1)r_{i^*}\notin \lbrack \bar{b} - \epsilon, \bar{b} + \epsilon\rbrack$. This contradicts that $r_{i^*} \leq 1/2 a_1$. Thus, we can assume that $kr_{i^*} < \bar{b} - \epsilon$ for all $k\in\lbrace 0,\ldots,s_{i^*}\rbrace$. 
	
	Since $\tilde{s}_{i^*} =\lfloor \sum_{j\in I} a_j / a_i^* \rfloor\leq  \lfloor \sum_{j=2}^na_j/a_i^*\rfloor = s_{i^*}$ we obtain the relation 
	\begin{align}
		\label{inequ_semgroup_proof_k=2_r_{i^*}}
		\tilde{s}_{i^*}r_{i^*} < \bar{b} - \epsilon.
	\end{align}
	By definition, we have $\bar{b} < a_1$. This bound can be strengthened to $\bar{b} < a_1 - \epsilon$ since, if $\bar{b}\geq a_1 - \epsilon$, we can approximate $b$ with an error of size at most $\epsilon$ by using $\lceil b/a_1 \rceil a_1$. Inserting (\ref{inequ_semgroup_proof_k=2_r_{i^*}}) and $\bar{b} < a_1 - \epsilon$ into (\ref{inequ_semigroup_proof_k=2}) results in
	\begin{align*}
		\sum_{i\in I}r_i - \tilde{s}_{i^*}r_{i^*}\leq \varphi(|I| - 1) \tilde{s}_{i^*}r_{i^*} < \varphi(|I| - 1)\left(\bar{b} - \epsilon\right) < \varphi(|I| - 1) (a_1 - 2\epsilon).
	\end{align*}
	Note that $\varphi(|I| - 1)\leq \varphi(n - 2)$ as $|I|\leq n- 1$. So we obtain
	\begin{align*}
		\sum_{i\in I}r_i - \tilde{s}_{i^*}r_{i^*} < \varphi(n - 2)(a_1 - 2\epsilon) = \varphi(n - 2)\left( 1 - \frac{2\varphi(n - 2)}{2\varphi(n - 2) + 1}\right)a_1 = \epsilon.
	\end{align*}
	Since $\bar{b} \leq \sum_{i\in I}r_i$, we deduce
	\begin{align*}
		\bar{b} - \tilde{s}_{i^*}r_{i^*} \leq\sum_{i\in I}r_i - \tilde{s}_{i^*}r_{i^*}< \epsilon.
	\end{align*}
	However, from (\ref{inequ_semgroup_proof_k=2_r_{i^*}}), we get $\tilde{s}_{i^*}r_{i^*} < \bar{b} - \epsilon$, which implies $\epsilon < \bar{b} - \tilde{s}_{i^*}r_{i^*}$. This contradicts the inequality above and proves the upper bound on $\app_{a_1,2}(\ba)$.

	\section{Instances with Large Approximation Error}
	\label{s_semigroups_lower_bounds}
	In this section we outline the lower bound constructions that are underlying the statements in Propositions~\ref{prop_semigroup_lower_bound_m_1_n-1}, \ref{prop_semigroup_lower_bound_m_1_n_2}, and \ref{prop_semigroup_lower_bound_general}. 
	Our proofs utilize the following number theoretic property.
	\begin{lemma}
		\label{lemma_semigroup_lower_bound}
		Let $\lambda_0,\ldots,\lambda_{n-1}\in\N$ be given such that $\sum_{i=0}^{n-1}\lambda_i2^i \equiv 2^n - 1 \mod 2^n$ and $\sum_{i=0}^{n-1}\lambda_i \leq n$. Then we have $\lambda_i = 1$ for all $i\in\lbrace 0,\ldots,n-1\rbrace$.
	\end{lemma}
	\begin{proof}
		Consider the involved numbers as binary numbers modulo $2^n$. The number $2^n-1$ is represented in binary as $1\ldots1$, a number with exactly $n$ digits being $1$. The number $2^i$ corresponds to $0\ldots010\ldots0$, a binary number with precisely one $1$ at the $i$-th digit. Given the congruence $\sum_{i=0}^{n-1}\lambda_i2^i \equiv 2^n - 1 \mod 2^n$, we have to reach the binary number $1\ldots1$ by adding binary numbers of the form $0\ldots010\ldots0$. Each addition with the latter number flips the digit at the $i$-th position. If we flip a $0$ to a $1$, the number of $1$'s in the new binary number increases by one. If we flip a $1$ to $0$, the number of $1$'s remains the same or even decreases due to carry-over effects. We only have $n$ flips available due to the assumption $\sum_{i=0}^{n-1}\lambda_i \leq n$. So the only way to reach the binary number with exactly $n$ digits being $1$ is by flipping each digit separately. This corresponds to $\lambda_i = 1$ for all $i\in\lbrace 0,\ldots,n-1\rbrace$.
	\end{proof}
	We begin by proving that the bounds in Theorem~\ref{thm_semigroup_m_1_n-1} are tight. 
	\begin{proof}[Proof of Proposition~\ref{prop_semigroup_lower_bound_m_1_n-1}]
		Let $a_1 := 2^{k}$ and $a_i := 2^{k+1} + 2^{i-2}$ for $i\in\lbrace 2,\ldots,k+1\rbrace$. For the remaining numbers, we suppose that $a_j > 2^{k+1}(k+1)$ for all $j\geq k+2$. We want to approximate 
		\begin{align*}
			b := \sum_{i=1}^{k+1}a_i = 2^{k+1}(k+1) -1
		\end{align*} 
		with $k$ numbers. 
		If $\bar{b}\in \N$ is a possible approximation with $b\neq\bar{b}$, we have $\left| b - \bar{b}\right| \geq 1 = 1/2^{k}\cdot a_1$. So it suffices to verify that it is impossible to represent $b$ using at most $k$ integers from $a_1,\ldots,a_n$. Since $a_j > 2^{k+1}(k+1) > b$, we can restrict ourselves to the first $k+1$ integers. 
		Let $\lambda_1,\ldots,\lambda_{k+1}\in \N$ be such that $\sum_{i=1}^{k+1}\lambda_i a_i = b$. We have $\sum_{i=2}^{k+1}\lambda_i \leq k$, because otherwise we get $\sum_{i=2}^{k+1}\lambda_i a_i > 2^{k+1}(k+1) > b$. We also have $\sum_{i=1}^{k+1}\lambda_i a_i \equiv b \mod 2^{k+1}$. This is equivalent to $\tilde{\lambda}_12^{k} + \sum_{i=2}^{k+1}\lambda_i 2^{i-2} \equiv 2^{k+1} - 1 \mod 2^{k+1}$, where $\tilde{\lambda}_1\in\lbrace 0,1\rbrace$ and $\tilde{\lambda}_1 + \sum_{i=2}^{k+1}\lambda_i\leq {k+1}$. By Lemma~\ref{lemma_semigroup_lower_bound}, we conclude that $\lambda_i = 1$ for all $i\in\lbrace 2,\ldots,k+1\rbrace$ and $\tilde{\lambda}_1 = 1$. Thus, there exists precisely one way to generate $b$ using $a_1,\ldots, a_{k+1}$, which is $b=a_1+\cdots+a_{k+1}$. So there is no exact representation that uses fewer than $k+1$ numbers.
	\end{proof}
	Next, we show that our bounds in Theorem~\ref{thm_semigroup_approx_2} are tight as well.
	\begin{proof}[Proof of Proposition~\ref{prop_semigroup_lower_bound_m_1_n_2}]
		For $n = 2$, the claim follows directly since $\varphi(n-2)$ is the empty sum which equals $0$. When $n=3$, we apply the tight bound from Proposition~\ref{prop_semigroup_lower_bound_m_1_n-1} and observe that $\varphi(1) = 1/2$, which implies $\varphi(1)/(2\varphi(1) + 1) = 1/4$. So let $n\geq 4$. 
		Observe that $\app_{2}(\ba) = \app_{2}(\kappa\cdot\ba)$ for some scalar $\kappa > 0$. The same holds for $\app_{a_1,2}(\ba)$. So we assume without loss of generality that $a_1 := 1$ and construct a rational vector $\ba\in\Q^n$ such that 
		\begin{align*}
			\frac{\varphi(n-2)}{2\varphi(n-2) + 1}\leq \app_2(\ba).
		\end{align*}
		Recall that $\lbrace t_i + 1\rbrace_{i\in \N}$ denotes the Sylvester sequence. Set $r_n := 1/(2\varphi(n-2) + 1)$ and define 
		\begin{align*}
			z_i := \prod_{\substack{j = 0, \\ j \neq n - i}}^{n-2}(t_j + 1) \text{ and } r_i := \frac{1}{t_{n-i}}\cdot r_n
		\end{align*}
		for $i\in\lbrace 2,\ldots,n\rbrace$. Let $\tau > 0$ be a fixed sufficiently large natural number. It will become clear from the arguments below what sufficient means in this context. We define $a_i := \tau\cdot z_i + r_i$ and try to approximate $b := a_2 + \cdots + a_n$. First, we approximate $b$ using $a_1$ and one more $a_i$ similar to the proof of Theorem~\ref{thm_semigroup_approx_2}. We define $s_i := \lfloor \sum_{j=2}^n a_j/a_i\rfloor$. So our candidates to approximate $a_2 + \cdots + a_n$ from below are $\mu a_i + \lambda a_1$ with some suitable $\lambda\in \N$ for $\mu\in\lbrace 0,\ldots, s_i\rbrace$ and $i\in\lbrace 2,\ldots, n\rbrace$. 
		Suppose that $\mu=0$. Then we approximate $b$ using only $a_1 = 1$. This is equivalent to rounding $r_2+\cdots+r_n$ to the nearest integer. We have
		\begin{align*}
			r_2+\cdots+r_n = r_n \sum_{i = 2}^{n}\frac{1}{t_{n-i}} = r_n(\varphi(n-2) + 1) = \frac{\varphi(n-2) + 1}{2\varphi(n-2) + 1}.
		\end{align*}
		Rounding this expression to the nearest integer gives us precisely an approximation error of  $\varphi(n-2)/(2\varphi(n-2) + 1)$. Hence, to improve upon this bound, we have to use another $a_i$. For that, we need an upper bound on $s_i$ to determine how often we can use each $a_i$. We have 
		\begin{align*}
			s_i = \left\lfloor \frac{\sum_{j=2}^n a_j}{a_i}\right\rfloor = \left\lfloor \frac{\tau \sum_{j=2}^n z_j}{\tau z_i + r_i} + \frac{\sum_{j=2}^nr_j}{\tau z_i + r_i}\right\rfloor \leq \left\lfloor \frac{\sum_{j=2}^n z_j}{z_i} + \frac{1}{\tau}c\right\rfloor,
		\end{align*}
		where we abbreviate $c = \sum_{j=2}^n r_j/z_i$. We continue the calculation and utilize the identity $\sum_{j=0}^{n-2}1/(t_{j} + 1) = 1 - 1/t_{n-1}$ from \cite{sylvestersequence1880sylvester} to obtain 
		\begin{align*}
			s_i \leq \left\lfloor \frac{\sum_{j=2}^n z_j}{z_i} + \frac{1}{\tau}c\right\rfloor &= \left\lfloor (t_{n-i} + 1)\cdot\sum_{j=2}^{n}\frac{1}{t_{n-j} + 1} + \frac{1}{\tau}c\right\rfloor \\
			& = \left\lfloor (t_{n-i} + 1)\cdot\sum_{j=0}^{n-2}\frac{1}{t_{j} + 1} + \frac{1}{\tau}c\right\rfloor \\
			& = \left\lfloor (t_{n-i} + 1)\cdot\left(1-\frac{1}{t_{n-1}}\right) + \frac{1}{\tau}c\right\rfloor \\
			& = \left\lfloor t_{n-i} + 1 - \frac{t_{n-i} + 1}{t_{n - 1}} + \frac{1}{\tau}c\right\rfloor = t_{n-i},
		\end{align*}
		where we use $t_{n-i} + 1 < t_{n-1}$ for all $i\in\lbrace 2,\ldots,n\rbrace$, as $n\geq 4$, and that $\tau$ is sufficiently large in the last step. We established $s_i\leq t_{n-i}$ for all $i\in\lbrace 2,\ldots,n\rbrace$. This implies $s_ir_i\leq r_n$ by the definition of $r_i$ for all $i\in\lbrace 2,\ldots,n\rbrace$. 
		We get 
		\begin{align*}
			r_2 + \cdots + r_n - s_ir_i \geq r_2 + \cdots + r_n - r_n = r_2 + \cdots + r_{n-1}.
		\end{align*}
		So the best approximation of $b$ using $\mu a_i + \lambda a_1$ with $\mu\geq 1$ gives us
		\begin{align*}
			\left| b - (\mu a_i + \lambda a_1)\right| \geq \left| r_2 + \cdots + r_n - s_ir_i\right|\geq \left| r_2 + \cdots + r_{n-1}\right|.
		\end{align*}
		The approximation error equals
		\begin{align*}
			r_2 + \cdots + r_{n-1} = r_n \cdot \sum_{j = 2}^{n-1} \frac{1}{t_{n-j}} = \frac{\varphi(n-2)}{2\varphi(n-2) + 1}.
		\end{align*}
		This is again our desired bound. So one cannot improve upon this bound by using $a_1$ and one of the other $a_i$. As a last step, we ensure that there is no other integral non-negative combination of at most two integers $a_i$ and $a_j$ where none of them equals $a_1$. Let $\lambda_i,\lambda_j\in \N$. Then we have
		\begin{align*}
			\left|b - \lambda_i a_i - \lambda_j a_j\right| = \left| \tau\left(\sum_{l=2}^{n}z_l - \lambda_i z_i - \lambda_j z_j\right) + \sum_{l=2}^n r_l - \lambda_ir_i - \lambda_jr_j\right|.
		\end{align*} 
		We abbreviate the last term with $\left|\tau d + e\right|$. We have $d\neq 0$ since $\lambda_i z_i + \lambda_j z_j$ is divisible by $(t_k + 1)$ for $t_k\notin\lbrace t_{n-i},t_{n-j}\rbrace$ but $\sum_{j=2}^{n}z_j$ is not. This implies that the error grows as $\tau$ becomes larger. This holds for all possible coefficients $\lambda_i$ and $\lambda_j$. Recall that we have $\lambda_i\leq s_i\leq t_{n-i}$ for all $i\in\lbrace 2,\ldots,n\rbrace$ if $\tau$ is large enough. So we only need to consider finitely many combinations of the form $(\lambda_i,\lambda_j)$ for $i,j\in\lbrace 2,\ldots,n\rbrace$. Hence, by choosing $\tau$ large enough, we increase the approximation error for any combination that can be formed without using $a_1$. This shows that $\varphi(n-2)/(2\varphi(n-2) + 1)\cdot a_1\leq \app_{2}(\ba)$.
	\end{proof}
	Our final step is to show that it is necessary to incorporate additional information from the matrix to obtain an approximation error that decreases when $k$ increases for $m\geq 2$.
	
	\begin{proof}[Proof of Proposition~\ref{prop_semigroup_lower_bound_general}]
		Let $q := \lfloor \sqrt{n-1}\rfloor$. Define 
		\begin{align*}
			\ba_1 := (n-1)\begin{pmatrix}
				1 \\ 0
			\end{pmatrix}, \ba_2 := (n-1)\begin{pmatrix}
				0 \\ 1
			\end{pmatrix}, \br := \begin{pmatrix}
				1 \\ q
			\end{pmatrix},
		\end{align*}
		and $\bB := (\ba_1,\ba_2)$. 
		Set $\ba_i := (n-1)\bz_i + \br$ with
		\begin{align*}
			\bz_i := (2^{n-2} + 1)\begin{pmatrix}
				1 \\ 1
			\end{pmatrix} +2^{i-2} \begin{pmatrix}
				1 \\ -1
			\end{pmatrix}
		\end{align*}
		for all $i\in\lbrace 3,\ldots,n\rbrace$. 
		The coefficient $2^{n-2} + 1$ is chosen such that each component of $\bz_i$ is positive for all $i\in\lbrace 3,\ldots,n\rbrace$. This guarantees $\ba_i \in\pos\lbrace \ba_1,\ba_2\rbrace$ for all $i\in\lbrace 3,\ldots,n\rbrace$. 
		We aim to approximate $\bb := \ba_1 + \cdots + \ba_n$ with some $\tilde{\bb}:=\lambda_1\ba_1 + \cdots + \lambda_n\ba_n$ for $\lambda_1,\ldots,\lambda_n\in \N$, where $n-k$ of the $\lambda_i$ are zero. 
		We distinguish between three cases based on the value of $\sum_{i=3}^n\lambda_i$: 
		\begin{enumerate}
			\item $\sum_{i=3}^n\lambda_i\leq n-3$,
			\item $\sum_{i=3}^n\lambda_i = n-2$, and 
			\item $\sum_{i=3}^n\lambda_i\geq n-1$.
		\end{enumerate}
		It will turn out that we only explicitly need the $q/(n-1)$ in the first case.
		In the other cases, the approximation error is even worse than $q/(n-1)$. 
		
		\textbf{Case 1:} Suppose that $\sum_{i=3}^n\lambda_i \leq n-3$.  
		For each $\by\in \R^2$, let $\lbrace\by\rbrace$ denote the unique vector from $\by + \bB\Z^2$ that is contained in the half-open parallelepiped $(n-1)\lbrack 0,1 )^2$. 
		In particular, we have $\lbrace\bb\rbrace = \lbrace (n-2)\br\rbrace$ and $\lbrace\tilde{\bb}\rbrace = \lbrace(\lambda_3+ \cdots+ \lambda_n)\br\rbrace$.  
		We claim that the choice of $\br$ implies 
		\begin{align}
			\label{inequ_prop_semigroup_proof}
			\Vert \lbrace(n-2)\br\rbrace - \lbrace\lambda\br\rbrace\Vert_{P(\bB)} \geq \frac{q}{n-1}
		\end{align}
		for all $\lambda\in\lbrace 0,\ldots, n-3\rbrace$. 
		In fact, we will show a slightly stronger statement: Let $\Lambda := (\ba_1,\ba_2,\br)\Z^3$. We prove that, if $\bx,\by\in \Lambda$ with $\bx\notin \by + \bB\Z^2$, then $\Vert \bx-\by\Vert_\infty\geq q$. This implies the claim above since $\lbrace(n-2)\br\rbrace\notin \lbrace\lambda\br\rbrace + \bB\Z^2$ and 
		$\Vert\bx\Vert_{P(\bB)} = 1/(n-1)\Vert\bx\Vert_\infty$ for all $\bx\in\R^2$. 
		
		Observe that minimizing $\Vert \bx-\by\Vert_\infty$ over $\bx,\by\in \Lambda$ with $\bx\notin \by + \bB\Z^2$ is equivalent to minimizing $\Vert \bz\Vert_\infty$ over $\bz\in\Lambda\backslash \bB\Z^2$. Any such lattice vector $\bz$ is contained in $t\br + \bB\Z^2$ for some $t\in\lbrack n-2\rbrack$. Let us fix such a lattice vector $\bz$. Suppose that the first coordinate of $\bz$ is non-negative, i.e., $z_1 \geq 0$, otherwise replace $\bz$ with $-\bz$. Since $t\in\lbrace 1,\ldots,n-2\rbrace$ and thus $t \neq 0$, we get $z_1 > 0$. If $z_1 \geq q$, we already have $\Vert\bz\Vert_\infty\geq q$. So suppose that $z_1 \in\lbrace 1,\ldots,q-1\rbrace$. This implies $t\in\lbrack q-1\rbrack$. 
		Consider the second coordinate of $\br$. We have
		\begin{align*}
			q \leq tr_2 \leq q^2 - q \leq n - 1 - q.
		\end{align*}
		Since $z_2 = tr_2 + s (n-1)$ for some integer $s\in\Z$, we obtain 
		\begin{align*}
			|z_2| = |tr_2 + s (n-1)|\geq q.
		\end{align*}
		Thus we conclude that $\Vert \bz\Vert_\infty\geq q$ and the inequality in (\ref{inequ_prop_semigroup_proof}) follows. Since adding multiples of $\ba_1$ and/or $\ba_2$ does not affect the estimate in (\ref{inequ_prop_semigroup_proof}), 
		we cannot achieve an approximation better than $q/(n-1)$.
		
		\textbf{Case 2:} Assume that $\sum_{i=3}^n\lambda_i = n-2$. 
		The vector $\tilde{\bb} - (\lambda_1\ba_1 + \lambda_2\ba_2)$ equals
		\begin{align*}
			(n-1)\left((2^{n-2} + 1) (n-2) \begin{pmatrix}
				1 \\ 1
			\end{pmatrix}
			+ \left(\sum_{i=3}^{n}\lambda_i2^{i-2}\right)\begin{pmatrix}
				1 \\ -1
			\end{pmatrix}\right) + (n-2)\br
		\end{align*}
		and $\bb$ equals
		\begin{align*}
			\ba_1 + \ba_2 + (n-1)\left((2^{n-2} + 1) (n-2)\begin{pmatrix}
				1 \\ 1
			\end{pmatrix} + 2\left(2^{n-2} - 1\right)\begin{pmatrix}
				1 \\ -1
			\end{pmatrix}\right) + (n-2)\br
		\end{align*}
		since the identity $\sum_{i=3}^n2^{i-2} = 2(2^{n-2} - 1)$ holds.
		We get the following difference
		\begin{align*}
			\tilde{\bb} - \bb = (\lambda_1 - 1)\ba_1 + (\lambda_2 - 1)\ba_2 + 2 (n-1)\underbrace{\left(\sum_{i=3}^{n}\lambda_i2^{i-3} - (2^{n-2} - 1)\right)}_{=:d}\begin{pmatrix}
				1 \\ -1
			\end{pmatrix}.
		\end{align*}
		Suppose that $d\neq 0$. We have $d\in\Z$ and assume that $d\geq 1$. Then the first coordinate of $\tilde{\bb}-\bb$ satisfies 
		\begin{align*}
			(\lambda_1 - 1)(n-1) + 2(n-1)d \geq n-1.
		\end{align*}
		Therefore, the approximation error can be estimated by $\Vert \bb - \tilde{\bb}\Vert_{P(\bB)}\geq 1 > q/(n-1)$. If $d\leq -1$, the same argument applies to the second coordinate. 
		We are left with the case when $d = 0$. Under this assumption, we have to have $\lambda_1 = 1 = \lambda_2$, otherwise one coordinate is again at least $n-1$ in absolute value. However, if $\lambda_1 = 1 = \lambda_2$ and $d=0$, we get $\lambda_i = 1$ for all $i\in\lbrack n\rbrack$ by Lemma~\ref{lemma_semigroup_lower_bound}, which is not a sparse representation.
		
		\textbf{Case 3:} We are left with the case $\sum_{i=3}^n\lambda_i \geq n-1$. By construction, the vectors $\ba_3,\ldots,\ba_n$ satisfy $\bm{1}^\top\ba_i = \bm{1}^\top \ba_j$ for $i,j\in\lbrace 3,\ldots,n\rbrace$. We use this fact to measure the distance between $\bb$ and $\tilde{\bb}$ as follows:
		\begin{align*}
			\bm{1}^\top\tilde{\bb} - \bm{1}^\top\bb \geq \bm{1}^\top (-\ba_1-\ba_2) + \bm{1}^\top\ba_3 \sum_{i=3}^n (\lambda_i - 1)\geq -2(n-1) + \bm{1}^\top\ba_3,
		\end{align*}
		where in the final inequality we apply the assumption $\sum_{i=3}^{n}\lambda_i\geq n-1$. Next we compute
		\begin{align*}
			\bm{1}^\top\ba_3 = (n-1)\bm{1}^\top\bz_3 + \bm{1}^\top\br =  2^{n-1}(n-1) + 2(n-1) + q + 1,
		\end{align*}
		which gives us $\bm{1}^\top\tilde{\bb} - \bm{1}^\top\bb\geq 2^{n-1}(n-1) + q + 1 > 2(n-1)$ as $n\geq 3$. We conclude with
		\begin{align*}
			\Vert \tilde{\bb} - \bb\Vert_{P(\bB)} = \frac{1}{n-1}\Vert \tilde{\bb} - \bb\Vert_\infty \geq \frac{1}{2(n-1)} \bm{1}^\top(\tilde{\bb} - \bb) > 1 > \frac{q}{n-1}
		\end{align*}
		using that $2\Vert \bx\Vert_\infty\geq\Vert\bx\Vert_1\geq \bm{1}^\top\bx$ for all $\bx\in\R^2$. 
		This shows that $\tilde{\bb}$ is too far away to be a decent approximation.
	\end{proof}	
	
	\bmhead{Acknowledgements} The authors are grateful to the reviewers of this manuscript and the reviewers of the extended abstract for their valuable suggestions and comments.
	
	\bmhead{Statements and Declarations}
	\begin{itemize}
		\item \textbf{Competing Interests/Funding:} The authors have no competing interests to declare that are relevant to the content of this article.
		\item \textbf{Availability of data and materials:} There is no data and materials information in an additional file.
	\end{itemize}

	\bibliographystyle{plain}
	\bibliography{references2}
	
	\appendix
	\section{Estimating Optimal Knapsack Values by Support One Solutions}
	
	In the proof of Theorem~\ref{thm_semigroup_approx_2}, we rely on \cite[Theorem~2]{kohlikrishna1992greedyknapsack}. This result is stated differently in \cite{kohlikrishna1992greedyknapsack} than it is applied in this paper: in particular, it is presented with the bound $1+\varphi(\infty)$ instead of $1 + \varphi(|I|-1)$. To address that the bound of $1+\varphi(\infty)$ can indeed be replaced by $1 + \varphi(|I|-1)$, one needs to make minor adjustments to the proof of Theorem~2 in \cite{kohlikrishna1992greedyknapsack}. For the sake of completeness, we present this adapted proof for the special integer knapsack problem \eqref{intknap} that appeared in the proof of Theorem~\ref{thm_semigroup_approx_2}. 
	
	Generally speaking, a knapsack problem \eqref{intknap} is of the following form. Let $r_1,\ldots,r_n > 0$ and $s_1,\ldots,s_n \in \N_{\geq 1}$. Consider
	\begin{equation}\label{intknap_appendix}
		\max \sum_{i=1}^nr_ix_i \ \text{ s.t. } \ \sum_{i = 1}^n \frac{1}{s_i + 1}x_i < 1, \ x_i\in\Z_{\geq 0} \text{ for all }i\in \lbrack n\rbrack.
	\end{equation}
	We claim that we obtain the following bound. 
	
	\begin{theorem}
		\label{thm_kk_variant}
		Let $z^*_1,\ldots,z^*_n$ be an optimal solution of \eqref{intknap_appendix}. Let $s_{i^*}r_{i^*} = \max_{i\in\lbrack n\rbrack} s_ir_i$ for some $i^*\in\lbrack n\rbrack$. Then we have
		\begin{align*}
			\sum_{i=1}^n r_iz^*_i \leq (1 + \varphi(n-1))  s_{i^*}r_{i^*}.
		\end{align*}
	\end{theorem}
	
	\begin{proof}
		We follow the arguments in \cite{kohlikrishna1992greedyknapsack}. 
		Consider the auxiliary integer knapsack problem
		
		\begin{equation}\label{intknap_proof_appendix}
			\max \sum_{i=1}^n\frac{1}{s_i}x_i \ \text{ s.t. } \ \sum_{i = 1}^n \frac{1}{s_i + 1}x_i < 1, \ x_i\in\Z_{\geq 0} \text{ for all }i\in \lbrack n\rbrack.
		\end{equation}
		Let $y^*_1,\ldots,y_n^*$ be an optimal solution of \eqref{intknap_proof_appendix}. It suffices to prove 
		
		\begin{align*}
			\frac{1}{s_{i^*}r_{i^*}}\cdot\sum_{i=1}^n r_iz^*_i \leq \sum_{i=1}^n \frac{1}{s_i}y^*_i \leq 1 + \varphi(n-1).
		\end{align*}
		
		We first verify the lower bound. By maximality of $s_{i^*}r_{i^*}$, we have $(s_{i}r_{i})/(s_{i^*}r_{i^*})\leq 1$ for all $i\in\lbrack n\rbrack$. Rearranging, yields $r_{i}/(s_{i^*}r_{i^*}) \leq 1/s_i$ for all $i\in\lbrack n\rbrack$. Since $z_1^*,\ldots,z_n^*$ is also feasible for \eqref{intknap_proof_appendix}, we get
		
		\begin{align*}
			\frac{1}{s_{i^*}r_{i^*}}\cdot\sum_{i=1}^n r_iz^*_i = \sum_{i=1}^n \frac{r_i}{s_{i^*}r_{i^*}}z^*_i \leq \sum_{i=1}^n \frac{1}{s_i}z^*_i\leq \sum_{i=1}^n \frac{1}{s_i}y^*_i,
		\end{align*}
		which proves the lower bound.
		
		We prove the upper bound by induction on $n$. For $n=1$, the bound holds as $\varphi(0) = 0$. Let $n \geq 2$. From the induction hypothesis, we deduce that $y^*_i\geq 1$ for all $i\in\lbrack n\rbrack$, otherwise, we can remove a variable and consider an integer knapsack problem with $n-1$ variables whose optimal value equals the optimal value of \eqref{intknap_proof_appendix} and is therefore at most $1 + \varphi(n-2) < 1 + \varphi(n-1)$. Let without loss of generality $s_1 < s_2 < \cdots < s_n$. The inequalities are strict because, if we have equalities, one can aggregate variables and use the induction hypothesis. We will show that we need $s_i = t_{i-1}$ for all $i\in\lbrack n-1\rbrack$ to get close to the upper bound. After that we will deal with $s_n$ separately. 
		
		We first prove that the optimal value is small when $s_1 \geq 2$. Fix $y_1^*$ and consider 
		
		\begin{align*}
			\max \frac{y_1^*}{s_1} + \sum_{i=2}^n\frac{1}{s_i}x_i \ \text{ s.t. } \ \sum_{i = 2}^n \frac{1}{s_i + 1}x_i < 1 - \frac{y_1^*}{s_1 + 1}, \ x_i\in\Z_{\geq 0}  \text{ for all }i\in \lbrack n\rbrack\backslash \lbrace 1\rbrace.
		\end{align*}
		Take an optimal vertex solution of the continuous relaxation with ``$<$'' replaced by ``$\leq$'' in the knapsack inequality, i.e., a vertex solution over the closure of the feasible region. Since $s_2 < s_i$ for all $i\geq 3$, this solution is obtained by having $x^*_2 = (s_2 + 1)(1-y_1^*/(s_1 + 1))$ and the other variables equal zero. The objective value for the continuous relaxation is a strict upper bound on the optimal value of the modified knapsack problem above. Using this and $s_1 + 1\leq s_2$, we calculate
		\begin{align*}
			\sum_{i=1}^n \frac{1}{s_i}y^*_i < \frac{y_1^*}{s_1} + \frac{x^*_2}{s_2} & = \frac{y_1^*}{s_1} + \left(1 + \frac{1}{s_2}\right)\cdot \left(1 - \frac{y_1^*}{s_1 + 1}\right) \nonumber \\ 
			& \leq \frac{y_1^*}{s_1} + \left(1 + \frac{1}{s_1+ 1}\right)\cdot \left(1 - \frac{y_1^*}{s_1 + 1}\right) \nonumber \\
			& = 1 + \frac{1}{s_1+ 1} + y^*_1 \left(\frac{1}{s_1} - \frac{1}{s_1 + 1} - \frac{1}{(s_1 + 1)^2}\right) \nonumber \\ 
			& = 1 + \frac{1}{s_1+ 1} + \frac{y_1^*}{s_1(s_1 + 1)^2} \\ 
			& \leq 1 + \frac{1}{s_1 + 1} + \frac{1}{(s_1 + 1)^2}
		\end{align*}
		as $y_1^*\leq s_1$. 
		Since $s_1 \geq 2$, we obtain 
		\begin{align*}
			\sum_{i=1}^n \frac{1}{s_i}y^*_i < 1 + \frac{1}{3} + \frac{1}{9} < 1 + \frac{1}{2} = 1 + \varphi(1) \leq 1 + \varphi(n-1)
		\end{align*}
		as $n\geq 2$. So the objective value of the optimal solution is strictly smaller than $1 + \varphi(n-1)$. Hence, for the remainder of the proof, we can assume that $s_1 = 1 = t_0$. 
		
		Our next goal is to show that $s_i = t_{i-1}$ for all $2\leq i\leq n-1$, otherwise the value of our optimal solution is too small. 
		Let $j^*\geq 2$ be the smallest index such that $s_{j^*} \neq t_{{j^*}-1}$. 
		We begin by showing that $y^*_i = 1$ for all $1\leq i\leq j^*-1$. By the induction hypothesis, we have $y^*_i \geq 1$. 
		It suffices to show that $y^*_i\leq 1$ for all $1\leq i \leq j^* - 1$. When $i = 1$, the claim follows from $y_1^*\leq s_1 = 1$. Let $i \geq 2$. The knapsack constraint gives us
		\begin{align*}
			\sum_{j=1}^{i-1}\frac{1}{t_{j-1} + 1} + \frac{y^*_{i}}{t_{i-1} + 1} < \sum_{j=1}^n \frac{1}{s_j}y^*_j < 1.
		\end{align*}
		Using the identity from \cite{sylvestersequence1880sylvester}, we get 
		\begin{align*}
			\frac{y^*_{i}}{t_{i-1} + 1} < 1 - \sum_{j=1}^{i-1}\frac{1}{t_{j-1} + 1} = \frac{1}{t_{i-1}},
		\end{align*}
		which implies $y^*_{i} \leq 1$ as $t_{i-1}\geq t_1 = 2$. So we established $y^*_i = 1$ for all $1\leq i\leq j^*-1$. Next, we distinguish two different cases, $j^* \leq n - 1$ and $j^* = n$. 
		
		Let $j^* \leq n - 1$. 
		Using that $y_{i}^*= 1$ for all $1\leq i \leq j^*-1$, one can estimate the knapsack constraint from below:
		\begin{align*}
			\sum_{i=1}^{j^*-1}\frac{1}{t_{i-1} + 1} + \frac{1}{s_{j^*} + 1} < \sum_{i=1}^n \frac{1}{s_i}y^*_i < 1.
		\end{align*}
		This is equivalent to 
		\begin{align*}
			\frac{1}{s_{j^*} + 1} < 1 - \sum_{i=1}^{j^*-1} \frac{1}{t_{i-1} + 1} = \frac{1}{t_{j^*-1}},
		\end{align*}
		where we again make use of the identity from \cite{sylvestersequence1880sylvester}. The inequality above implies that $s_{j^*}\geq t_{j^*-1}$. Since $s_{j^*}\neq t_{j^*-1}$ by assumption, we get $s_{j^*}\geq t_{j^*-1} + 1$. We will use this to show that the optimal value of \eqref{intknap_proof_appendix} is smaller than $1+\varphi(n-1)$. Similar to the case, where we showed that we can assume $s_1 = 1$, fix $y_1^*,\ldots,y_{j^*-1}^*$ and consider the restricted integer knapsack problem with constraint
		\begin{align*}
			\sum_{i = j^*}^n \frac{1}{s_i + 1}x_i < 1 - \sum_{i=1}^{j^*-1} \frac{1}{t_{i-1} + 1}y^*_{i} =  1 - \sum_{i=1}^{j^*-1} \frac{1}{t_{i-1} + 1} = \frac{1}{t_{j^*-1}},
		\end{align*} 
		where we use $y^*_{i} = 1$ and \cite{sylvestersequence1880sylvester}. 
		Again, we pick the optimal vertex solution for the continuous problem over the closure of the feasible region, which is given by the only non-zero entry $x^*_{j^*} = (s_{j^*} + 1)/t_{j^*-1}$. We estimate the value of an optimal solution by the continuous relaxation and use $s_{j^*}\geq t_{j^*-1} + 1$, which we established above, to obtain
		\begin{align*}
			\sum_{i=1}^n \frac{1}{s_i}y^*_i < \sum_{i=1}^{j^*-1}\frac{1}{t_{i-1}} + \frac{x^*_{j^*}}{s_{j^*}} & = (1 + \varphi(j^* - 2)) + \left(\frac{1}{t_{j^*-1}} + \frac{1}{t_{j^*-1}s_{j^*}}\right) \\ 
			& \leq 1 + \varphi(j^* - 2) + \frac{1}{t_{j^*-1}} + \frac{1}{t_{j^*-1}(t_{j^*-1} + 1)} \\
			& = 1 + \varphi(j^*) \leq 1 + \varphi(n-1).
		\end{align*} 
		This solves the case when $j^* \leq n-1$.
		
		Let $j^* = n$. As before, fix the variables $y_1^*,\ldots,y_{n-1}^*$, which are all equal one, and consider the restricted integer knapsack constraint
		
		\begin{align*}
			\frac{1}{s_n + 1}x_n < 1 - \sum_{i=1}^{n-1} \frac{1}{t_{i-1} + 1}y^*_{i} =  1 - \sum_{i=1}^{j^*-1} \frac{1}{t_{i-1} + 1} = \frac{1}{t_{j^*-1}}.
		\end{align*} 
		We have $y^*_n = (s_n + 1 - l)/t_{j^*-1}\in\N$ for some $l\in\lbrace 1,\ldots,t_{j^*-1}\rbrace$. Plugging this into the objective function yields
		
		\begin{align*}
			\sum_{i=1}^n \frac{1}{s_i}y^*_i = 1 + \varphi(n-2) + \frac{s_n + 1 - l}{s_n t_{j^*-1}}\leq 1 + \varphi(n-2) + \frac{1}{t_{j^*-1}} = 1 + \varphi(n-1),
		\end{align*}
		where the last inequality comes from $l\geq 1$.
	\end{proof}
	
	While this proof slightly differs from the proof of \cite[Theorem~2]{kohlikrishna1992greedyknapsack}, we recycled some key arguments, in particular, the technique of estimating the objective value using the continuous relaxation. The notable differences are the argument in the case $j^* = n$, not relying on the significantly larger knapsack problem $P^*$ in \cite{kohlikrishna1992greedyknapsack}, and abbreviating some arguments by using induction on $n$. The proof above combined with the discussion in \cite{kohlikrishna1992greedyknapsack} at the beginning of Section~2 generalizes to arbitrary knapsack problems. We refrain from discussing this here since we only need the special problem stated in \eqref{intknap_appendix}.
	
\end{document}